\documentclass[12pt]{article}

 \usepackage[vmargin=1.18in,hmargin=1.1in]{geometry}
 \usepackage{amssymb,amsfonts,graphicx,amsthm,nicefrac,mathtools,bm, bbm}
 \usepackage[hidelinks]{hyperref}
 \usepackage[normalem]{ulem} 
 \usepackage{mathrsfs}
 \usepackage{pdflscape}

 \usepackage[dvipsnames]{xcolor}
 \usepackage{graphics,tikz}

\usepackage[]{caption}

\usepackage[linesnumbered,ruled]{algorithm2e}
\let\oldnl\nl
\newcommand{\nonl}{\renewcommand{\nl}{\let\nl\oldnl}}
\makeatother

\usepackage[uniquename=false, uniquelist = false, url = false, doi = false, isbn = false, sorting = none, natbib = true, backend = bibtex, maxbibnames=10]{biblatex}

 \newtheorem{theorem}{Theorem}
 \newtheorem{proposition}{Proposition}
 
 \newtheorem{lemma}{Lemma}
 \newtheorem{corollary}{Corollary}

 \theoremstyle{definition}
 
   \newtheorem{setting}{Setting}
 \theoremstyle{remark}

 \def\independenT#1#2{\mathrel{\rlap{$#1#2$}\mkern2mu{#1#2}}}
 \newcommand\independent{\protect\mathpalette{\protect\independenT}{\perp}}

 \newcommand{\ignore}[1]{}

 \newcommand{\eps}{\epsilon}

 \newcommand{\prx}{\bm X} 
  
 \newcommand{\srx}{X}
 
 \newcommand{\sfx}{x}

\newcommand{\prz}{\bm Z} 
  
\newcommand{\srz}{Z}

\newcommand{\sfz}{z}

 \newcommand{\prxk}{{{\widetilde{\bm X}}}}
 \newcommand{\srxk}{\widetilde X}
 \newcommand{\sfxk}{\widetilde x}
 
 \newcommand{\pry}{{\bm Y}}

 \newcommand{\sry}{Y}
 
 \newcommand{\sfy}{y}
 
\newcommand{\seps}{\epsilon}
\newcommand{\peps}{\bm \epsilon}

  \def\CRT{\textnormal{CRT}}
 
\begin{document}

\title{On the power of conditional independence testing \\ under model-X}
\author{Eugene Katsevich$^1$ and Aaditya Ramdas$^{1,2}$\\
	[10pt]
	\texttt{ekatsevi@wharton.upenn.edu},
	\texttt{aramdas@stat.cmu.edu} \\
	Department of Statistics and Data Science, University of Pennsylvania$^1$\\
	Department of Statistics and Data Science, 	Carnegie Mellon University$^2$\\
	Machine Learning Department, Carnegie Mellon University$^2$
}

\maketitle
\thispagestyle{empty}

\begin{abstract}
For testing conditional independence (CI) of a response $Y$ and a predictor $X$ given covariates $Z$, the recently introduced model-X (MX) framework has been the subject of active methodological research, especially in the context of MX knockoffs and their successful application to genome-wide association studies. In this paper, we study the power of MX CI tests, yielding quantitative insights into the role of machine learning and providing evidence in favor of using likelihood-based statistics in practice. Focusing on the conditional randomization test (CRT), we find that its conditional mode of inference allows us to reformulate it as testing a point null hypothesis involving the conditional distribution of $X$. The Neyman-Pearson lemma then implies that a likelihood-based statistic yields the most powerful CRT against a point alternative. We also obtain a related optimality result for MX knockoffs. Switching to an asymptotic framework with arbitrarily growing covariate dimension, we derive an expression for the limiting power of the CRT against local semiparametric alternatives in terms of the prediction error of the machine learning algorithm on which its test statistic is based. Finally, we exhibit a resampling-free test with uniform asymptotic Type-I error control under the assumption that \textit{only the first two moments of $X$ given $Z$ are known}, a significant relaxation of the MX assumption. 
\end{abstract}
	


\color{black}

\section{Introduction}

\subsection{Conditional independence testing and the MX assumption}

Given a predictor $\prx \in \mathbb R^{d}$, response $\pry \in \mathbb R^{r}$, and covariate vector $\prz \in \mathbb R^{p}$ drawn from a joint distribution $(\prx, \pry, \prz) \sim \mathcal L$, consider testing the hypothesis of conditional independence (CI),
\begin{equation}
H_0: \pry \independent \prx\ |\ \prz \quad \text{versus} \quad H_1: \pry \not \independent \prx\ |\ \prz,
\label{conditional-independence}
\end{equation}
using $n$ data points 
\begin{equation}
\label{eq:xyz}
(\srx, \sry, \srz) \equiv \{(\srx_i, \sry_i, \srz_i)\}_{i = 1, \dots, n} \overset{\text{i.i.d.}}\sim \mathcal L. 
\end{equation}
This fundamental problem---determining whether a predictor is associated with a response after controlling for a set of covariates---is ubiquitous across the natural and social sciences. To keep an example in mind throughout the paper, consider $\pry \in \mathbb R^1$ cholesterol level, $\prx \in \{0,1,2\}^{10}$ the genotypes of an individual at 10 adjacent polymorphic sites, and $\prz \in \{0,1,2\}^{500,000}$ the genotypes of the individual at other polymorphic sites across the genome. Such data $(\srx, \sry, \srz)$ would be collected in a genome-wide association study (GWAS), with the goal of testing for association between the 10 polymorphic sites of interest and cholesterol while controlling for the other polymorphic sites~\eqref{conditional-independence}. CI testing is also connected to causal inference: with appropriate unconfoundedness assumptions, Fisher's sharp null hypothesis of no effect of a (potentially non-binary) treatment $\prx$ on an outcome $\pry$ implies conditional independence. While we do not work in a causal framework, we draw inspiration from connections to causal inference throughout.

As formalized by Shah and Peters \cite{Shah2018}, the problem \eqref{conditional-independence} is fundamentally impossible without assumptions on the distribution $\mathcal L(\prx,\pry, \prz)$, in which case no asymptotically uniformly valid test of this hypothesis can have nontrivial power against \emph{any} alternative. In special cases, the problem is more tractable, for example if $\prz$ has discrete support, or if we were willing to make (semi)parametric assumptions on the form of $\mathcal L(\pry|\prx, \prz)$ (henceforth ``model-$Y|X$''). We will not be making such assumptions in this work.
Instead, we follow the lead of Candes et al. \cite{CetL16}, who proposed to avoid assumptions on $\mathcal L(\pry|\prx, \prz)$, but assume that we have access to $\mathcal L(\prx|\prz)$:\footnote{Candes et al. actually require that the full joint distribution $\mathcal L(\prx,\prz)$ is known, but this is because they also test for conditional associations between $\prz$ and $\pry$. We focus only on the relationship between $\prx$ and $\pry$ given $\prz$ and therefore require a weaker assumption.} 
\begin{equation} \label{eq:modelX}
\textit{model-X (MX) assumption}: \mathcal L(\prx| \prz)=  f^*_{\prx|\prz} \text{ for some known } f^*_{\prx|\prz}.\footnote{We implicitly assume that $\mathcal L$ has a density with respect to some dominating measure on $\mathbb R^{1+1+p}$, and that all conditional densities are well-defined almost surely. Here and throughout the paper, we identify probability distributions with their densities with respect to the appropriate dominating measure.}
\end{equation}
Candes et al.\ argue that while both model-$Y|X$ and MX are strong assumptions---especially when $p,d$ are large---in certain cases much more is known about $\prx | \prz$ than about $\pry|\prx, \prz$. In the aforementioned GWAS example, $\prx|\prz$ reflects the joint distribution of genotypes at SNPs across the genome, which is well described by hidden Markov models from population genetics \cite{SetC17}. On the other hand, the distribution $\pry|\prx,\prz$ represents the genetic basis of a complex trait, about which much less is known. In the context of (stratified) randomized experiments, the distribution $\mathcal L(\prx | \prz)$ is the propensity function \cite{Imai2004} (the analog of the propensity score for non-binary treatments \cite{Rosenbaum1983}) and is experimentally controlled. In general causal inference contexts, the MX assumption can be viewed as the assumption that the propensity function is known. 

\subsection{MX methodology and open questions}

Testing CI hypotheses in the MX framework has been the subject of active methodological research. The most popular methodology is MX knockoffs \cite{CetL16}. This method is based on the idea of constructing synthetic negative controls (knockoffs) for each predictor variable in a rigorous way that is based on the MX assumption; see Section~\ref{sec:knockoffs-overview} for a brief overview. Rapid progress has been made on the construction of knockoffs in various cases \cite{SetC17,Romano2019a,Bates2019,Huang2019} and on the application of this methodology to GWAS \cite{SetC17, SetS19}. The conditional randomization test (CRT) \cite{CetL16}, initially less popular than knockoffs due to its computational cost, is receiving renewed attention as computationally efficient variants are proposed, such as the holdout randomization test (HRT) \cite{Tansey2018}, the digital twin test \cite{Bates2020}, and the distilled CRT (dCRT) \cite{Liu2020}. The dCRT in particular is a promising methodology because it combines good power and computational speed; we focus on this variant of the CRT is Sections~\ref{sec:weakening} and~\ref{sec:asymptotic-power} of this paper. We introduce the general CRT methodology next, while deferring the introduction of the dCRT to Section~\ref{sec:weakening}. 

We start with any test statistic $T(\srx, \sry, \srz)$ measuring the association between $\prx$ and $\pry$, given $\prz$. Usually, this statistic involves learning some estimate $\widehat f_{\pry|\prx,\prz}$ based on machine learning, e.g. the magnitude of the fitted coefficient for $\prx$ (when $\text{dim}(\prx) = 1$) in a cross-validated lasso~\cite{T96} of $\sry$ on $\srx$ and $\srz$ \cite{CetL16}. To calculate the distribution of $T$ under the null hypothesis~\eqref{conditional-independence}, first define a matrix  $\srxk \in \mathbb R^{n \times d}$, where the $i$th row $\srxk_i$ is a sample from $\mathcal L(\prx \mid \prz = \srz_i)$. In other words, for each sample $i$, resample $\srx_i$ based on its distribution conditional on the observed covariate values $Z_i$ in that sample. We then use these resamples to build a null distribution $T(\srxk, \sry, \srz)$, from which the upper quantile
\begin{equation}
C(\sry,\srz) \equiv Q_{1-\alpha}[T(\srxk, \sry, \srz)|\sry,\srz]
	\label{upper-quantile}
\end{equation}
may be extracted (the dependence on $\alpha$ left implicit), where the randomness is over the resampling distribution $\srxk |  \sry, \srz$. Finally, the CRT rejects if the original test statistic exceeds this quantile:

\begin{equation}
	\label{eq:randomized-CRT}
	\phi^{\CRT}_T(\srx,\sry,\srz) \equiv 
	\begin{cases}
		1, \quad &\text{if }  T(\srx, \sry, \srz) > C(\sry,\srz); \\
		\gamma, \quad &\text{if }  T(\srx, \sry, \srz) = C(\sry,\srz); \\
		0, \quad &\text{if }T(\srx, \sry, \srz) < C(\sry,\srz).
	\end{cases}
\end{equation}
In order to accommodate discreteness, the CRT makes a randomized decision $\gamma$ when $T(\srx, \sry, \srz) = C(\sry,\srz)$ so that the size of the test is exactly $\alpha$.
In practice, the threshold $C(\sry,\srz)$ is approximated by computing $T(\srxk^b, \sry, \srz)$ for a large number $B$ of Monte Carlo resamples $\srxk^b \sim \srx|\srz$. 
For the sake of clarity, this paper considers only the ``infinite-$B$" version of the CRT as defined by~\eqref{upper-quantile} and \eqref{eq:randomized-CRT}. In the causal inference setting, the CRT can be viewed as a variant of Fisher's exact test for randomized experiments that incorporates strata of covariates \cite{Zheng2008,Hennessy2016}, basing inference on rerandomizing the treatment to the units. 

The CI testing problem under MX has benefited from several methodological innovations, but fundamental questions regarding power and optimality have received less attention. Therefore, in this paper we address the following two primary questions:
\begin{enumerate}
	\item[Q1.] Are there ``optimal" test statistics for MX methods, in any sense?
	\item[Q2.] What is the precise connection between the performance of the machine learning algorithm and the power of the resulting MX method?
\end{enumerate}
To the best of our knowledge, Q1 has not been considered before, while Q2 has only been indirectly addressed in the context of lasso-based knockoffs \cite{Weinstein2017, Liu2019, Fan2020, Weinstein2020} and CRT \cite{Wang2020b, Celentano2020}. The present paper complements these existing works by considering arbitrary machine learning methods. We summarize our findings next.

\subsection{Our contributions}

We find that for the MX CI problem, the CRT is more natural to analyze; it is simpler to analyze than MX knockoffs and is applicable for testing even a single conditional independence hypothesis. Thus, we focus mainly on the CRT in the present paper. We obtain the following nontrivial answers to the questions posed above.

\paragraph{A1: Conditional inference leads to finite-sample optimality against point alternatives.}

While the composite nonparametric alternative of the CI problem~\eqref{conditional-independence} suggests that we cannot expect to find a uniformly most powerful test, we may still ask what is the most powerful test against a point alternative. Restricting our attention to tests valid conditionally on $(\sry, \srz)$ (as the CRT is) allows us to reduce the composite null to a point null. We can therefore apply the Neyman-Pearson lemma to show (Section~\ref{sec:power}) that the optimal conditionally valid test against a point alternative $\mathcal L$ with $\mathcal L(\pry|\prx,\prz) = \bar f_{\pry|\prx,\prz}$ is the CRT based on the likelihood test statistic:
\begin{equation}
	T^{\text{opt}}(\srx; \sry, \srz) \equiv \prod_{i = 1}^n \bar f(\sry_i|\srx_i, \srz_i).
\end{equation}
The same statistic yields the most powerful one-bit $p$-values for MX knockoffs (Section~\ref{sec:knockoffs}). Despite the simplicity of this result, it has not been derived before and appears central to the design of powerful test statistics. Since the model for $Y|X,Z$ is unknown, this result provides our first theoretical indication of the usefulness of machine learning models to learn this distribution (Q2). A2 below gives a more quantitative answer to Q2.

\paragraph{A2: The prediction error of the machine learning method impacts the asymptotic efficiency of the dCRT but not its consistency.}

It has been widely observed that the better the machine learning method approximates $\pry|\prx,\prz$, the higher power the MX method will have. We put this empirical knowledge on a theoretical foundation by expressing the asymptotic power of the dCRT in terms of the prediction error of the underlying machine learning method (Section~\ref{sec:asymptotic-power}). In particular, we consider semiparametric alternatives of the form
\begin{equation}
	H_1: \mathcal L(\pry|\prx,\prz) = N(\prx^T\beta + g(\prz),\sigma^2).
	\label{parametric-alternative-intro}
\end{equation}
We analyze the power of a dCRT variant that employs a separately trained estimator $\widehat g$ in an asymptotic regime where $d = \text{dim}(\prx)$ remains fixed while $p = \text{dim}(\prz)$ grows arbitrarily with the sample size $n$. We find that this test is consistent no matter what $\widehat g$ is used, while its asymptotic power against local alternatives $\beta_n = h/\sqrt{n}$ depends on the limiting mean-squared prediction error of $\widehat g$ (denoted $\mathcal E^2$) and the limiting expected variance $\mathbb E[\text{Var}[\prx | \prz]]$ (denoted $s^2$). For example, if $d = 1$, the
\begin{equation*}
	\begin{split}
		\text{dCRT power converges to that of normal location test under alternative } N\smash{\left(\frac{hs}{\sqrt{\sigma^2 + \mathcal E^2}},1\right)}.
	\end{split}
\end{equation*}
This represents the first explicit quantification of the impact of machine learning prediction error on the power of an MX method. \\

On the way to addressing Q2, we additionally establish a third result (Section~\ref{sec:weakening}) that may be of independent interest:

\paragraph{A resampling-free second-order approximation to the dCRT is equivalent to the dCRT and controls Type-I error under weaker assumptions.}

It was recently pointed out that if $\mathcal L(\prx|\prz)$ is Gaussian, then the resampling distribution of the dCRT test statistic can be found in closed form without actual resampling \cite{Liu2020}. Here we show that the resampling-free dCRT based on the first two moments of $\mathcal L(\prx|\prz)$ is asymptotically equivalent to the dCRT based on $\mathcal L(\prx|\prz)$ itself. Furthermore, we show the former test has asymptotic Type-I error control under the
\begin{equation}
	\begin{split}
		&\textit{MX(2) assumption:} \text{ the first two moments of $\prx|\prz$ are known, i.e.} \\
		&\mathbb E_{\mathcal L}[\prx|\prz] = \mu(\prz) \text{ and } \text{Var}_{\mathcal L}[\prx|\prz] = \Sigma(\prz) \text{ for known } \mu(\cdot), \Sigma(\cdot).
		\label{MX(2)-intro}
	\end{split}
\end{equation}
This assumption is weaker than the full MX assumption, complementing existing work \cite{Huang2019, Barber2020} on weakening assumptions for MX methods. It also suggests that the resampling-free dCRT may be used in place of the usual dCRT while achieving similar power and controlling Type-I error asymptotically.

\paragraph*{}

These advances shed new light on the nature of the MX problem and can inform methodological design. Our results handle multivariate $\prx$, arbitrarily correlated designs in the model for $\prx$, and any black-box machine learning method to learn $\widehat g$. 
\paragraph{Notation.}

Recalling equations~\eqref{conditional-independence} and \eqref{eq:xyz}, population-level variables (such as $\prx,\pry,\prz$) are denoted in boldface, while samples of these variables (such as $\srx_i,\sry_i,\srz_i$) are denoted in regular font. Note that boldface does \textit{not} distinguish between scalars, vectors, and matrices, as it is sometimes employed. The dimensions of the object in this paper will be clear from context. All vectors are treated as column vectors. 
We often use uppercase symbols to denote both random variables and their realizations (for either population- or sample-level quantities), but use lowercase to denote the latter when it is important to make this distinction. We use $\mathcal L$ to denote the joint distribution of $(\prx,\pry,\prz)$, though we sometimes use this symbol to denote the joint distribution of $(\srx,\sry,\srz)$ as well. We use the symbol ``$\equiv$" for definitions. We denote by $c_{d,1-\alpha}$ the $1-\alpha$ quantile of the $\chi^2_d$ distribution, and by $\chi^2_d(\lambda)$ the non-central $\chi^2$ distribution with $d$ degrees of freedom and noncentrality parameter $\lambda$.

\section{The most powerful CRT against point alternatives} \label{sec:power}

In this section, we seek the most powerful CRT against a point alternative. To accomplish this, we make the observation---implicit in earlier works---that the CRT is valid not just unconditionally but also conditionally on $\sry, \srz$ (Section~\ref{sec:bridge}). The latter conditioning step reduces the composite null to a point null. This reduction allows us to invoke the Neyman Pearson lemma to find the most powerful test  (Section~\ref{sec:NP}). Proofs are deferred to the appendix.

\subsection{CRT is conditionally valid and implicitly tests a point null} \label{sec:bridge}

Let us first formalize the definition of a level $\alpha$ test of the MX CI problem. The null hypothesis is defined as the set of joint distributions compatible with conditional independence and with the assumed model for $\prx|\prz$:
\begin{equation}
	\label{eq:null-under-modelX}
	\begin{split}
		\mathscr L^{\text{MX}}_0(f^*) &\equiv \mathscr L_0 \cap \mathscr L^{\text{MX}}(f^*) \\
		&\equiv \{\mathcal L: \prx \independent \pry \mid \prz\} \cap \{\mathcal L: \mathcal L(\prx|\prz) = f^*_{\prx|\prz}\}\\
		&= \{\mathcal L: \mathcal L(\prx,\pry,\prz) = f_{\prz} \cdot f^*_{\prx|\prz} \cdot f_{\pry|\prz} \text{ for some } f_{\prz}, f_{\pry|\prz} \}.
	\end{split}
\end{equation}
A test $\phi: (\mathbb R^{d} \times \mathbb R^r \times \mathbb R^p)^n \rightarrow [0,1]$ of the MX CI problem is said to be level $\alpha$ if
\begin{equation}
	\sup_{\mathcal L \in \mathscr L^{\text{MX}}_0(f^*)} \mathbb E_{\mathcal L}[\phi(\srx, \sry, \srz)] \leq \alpha.
	\label{marginal-validity}
\end{equation}
Recall that the CRT critical value $C(\sry, \srz)$ is defined via conditional calibration~\eqref{upper-quantile}. As is known to those familiar with MX, this implies that any CRT $\phi = \phi_T^{\text{CRT}}$ not only has level $\alpha$ in the sense of definition~\eqref{marginal-validity} but also has level $\alpha$ \textit{conditionally} on $\sry$ and $\srz$:
\begin{equation}
\sup_{\mathcal L \in \mathscr L_0^{\textnormal{MX}}(f^*)}\ \mathbb E_{\mathcal L}[\phi(\srx, \sry, \srz)|\sry, \srz] \leq \alpha \quad \text{almost surely}.
\label{eq:conditional-validity}
\end{equation}

One special property of such conditionally valid tests $\phi$ is that they can be viewed as testing a \textit{point null} rather than the original \textit{composite null} \eqref{conditional-independence}. To see this, we view $\phi \equiv \phi(\srx;\sry,\srz)$ as a \textit{family} of hypothesis tests, indexed by $(\sry, \srz)$, for the distribution $\mathcal L(\srx|\sry,\srz)$. Note that under the MX assumption,
\small
\begin{equation}
\mathcal L \in \mathscr L_0^{\text{MX}}(f^*) \Longrightarrow \mathcal L(\srx = \sfx|\sry = \sfy, \srz = \sfz) = \prod_{i = 1}^n  f^*(\sfx_i | \sfz_i).
\label{point-null}
\end{equation}
\normalsize
In words, fixing $Y,Z$ at their realizations $\sfy,\sfz$ and viewing only $\srx$ as random, $\mathcal L(\srx|\sry = \sfy, \srz = \sfz)$ equals a fixed product distribution for any null $\mathcal L$. This yields a conditional point null hypothesis, with respect to which $\phi^{\CRT}_T(\sfx;\sfy,\sfz)$ is a level-$\alpha$ test for almost every $(\sfy, \sfz)$. Note that the observations $\srx_i$ in this conditional distribution are independent \textit{but not identically distributed} due to the different conditioning events in~\eqref{point-null}.

We emphasize that the aforementioned observations have been under the hood of MX papers, and the existence of a single null distribution from which to resample $\srxk$ is central to the very definition of the CRT. Nevertheless, we find it useful to state explicitly what has thus far been largely left implicit. Indeed, viewing the CRT through the conditional lens~\eqref{eq:conditional-validity} is the starting point that allows us to bring classical theoretical tools to bear on its analysis. We start doing so by considering point alternatives below.

\subsection{The most powerful conditionally valid test} \label{sec:NP}

Viewing the CRT as a test of a point null hypothesis, we can employ the  Neyman-Pearson lemma to find the most powerful CRT (in fact, the most powerful conditionally valid test) against point alternatives. The following theorem states that the likelihood ratio with respect to the (unknown) distribution $\pry|\prx, \prz$ is the most powerful CRT test statistic against a point alternative.

\begin{theorem} \label{prop:crt-optimality}
	Let $\bar{\mathcal L} \in \mathscr L^{\text{MX}}(f^*)$ be an alternative distribution, with $\bar{\mathcal L}(\pry|\prx,\prz) = \bar f_{\pry|\prx,\prz}$. The likelihood of the data $(\srx, \sry, \srz)$ with respect to $\bar{\mathcal L}(\pry|\prx,\prz)$ is
	\begin{equation}
	T^{\textnormal{opt}}(\srx, \sry, \srz) \equiv \prod_{i = 1}^n  \bar f(\sry_i|\srx_i, \srz_i).
	\label{log-likelihood-ratio}
	\end{equation}
	The CRT $\phi^{\CRT}_{T^{\textnormal{opt}}}$ based on this test statistic is the most powerful conditionally valid test of $H_0: \mathcal L \in \mathscr L_0^{\textnormal{MX}}(f^*)$ against $H_1: \mathcal L = \bar{\mathcal L}$, i.e. 
	\begin{equation}
	\mathbb E_{\bar{\mathcal L}}[\phi(\srx,\sry,\srz)] \leq \mathbb E_{\bar{\mathcal L}}[\phi^{\CRT}_{T^\textnormal{opt}}(\srx, \sry, \srz)]
	\label{unconditional}
	\end{equation}
	for any test $\phi$ satisfying the conditional validity property~\eqref{eq:conditional-validity}.
\end{theorem}

We leave open the question of whether $\phi^{\CRT}_{T^\textnormal{opt}}$ is also the most powerful test among not just conditionally valid tests~\eqref{eq:conditional-validity} but also among marginally valid tests~\eqref{marginal-validity}. There do at least exist marginally valid tests that are not conditionally valid.

The proof of Theorem~\ref{prop:crt-optimality} (Appendix~\ref{sec:proofs-sec2}) is based on the reduction in Section~\ref{sec:bridge} of the composite null to a point null by conditioning, followed by the Neyman-Pearson lemma. Note that the likelihood ratio in the model $\mathcal L(\prx|\pry, \prz)$ reduces to the likelihood in the model $\mathcal L(\pry|\prx, \prz)$ up to constant factors; see derivation~\eqref{eq:likelihood-ratio-derivation}. This argument has similar flavor to the theory of unbiased testing (see Lehmann and Romano \cite[Chapter 4]{TSH}), where uniformly most powerful unbiased tests can be found by conditioning on sufficient statistics for nuisance parameters. Our result is also analogous to but different from Lehmann's derivation of the most powerful permutation tests using conditioning followed by the Neyman-Pearson lemma, in randomization-based causal inference (see the rejoinder of Rosenbaum's 2002 discussion paper \cite{Rosenbaum2002}, Section 5.10 of Lehmann (1986), now Lehmann and Romano \cite[Section 5.9]{TSH}).

Inspecting the most powerful test given by Theorem~\ref{prop:crt-optimality}, we find that it depends on $\bar{\mathcal L}$ only through $\bar{\mathcal L}(\pry|\prx,\prz)$. This immediately yields the following corollary.

\begin{corollary} \label{cor:crt-optimality}
	Define the composite class of alternatives
	\begin{equation*}
	\begin{split}
	\mathscr L_1(f^*, \bar f) &= \{\mathcal L \in \mathscr L_0^{\textnormal{MX}}(f^*): \bar{\mathcal L}(\pry|\prx,\prz) = \bar f_{\pry|\prx,\prz}\} \\
	&= \{\mathcal L: \mathcal L(\prx,\pry,\prz) = f_{\prz} \cdot f^*_{\prx|\prz} \cdot \bar f_{\pry|\prx,\prz} \textnormal{ for some } f_{\prz}\}.
	\end{split}
	\end{equation*}
	Among the set of conditionally valid tests~\eqref{eq:conditional-validity}, the test $\varphi^{\CRT}_{T^{\textnormal{opt}}}$ is uniformly most powerful against $\mathscr L_1(f^*, \bar f)$.
\end{corollary}

Theorem~\ref{prop:crt-optimality} and Corollary~\ref{cor:crt-optimality} imply that the most powerful CRT against a point alternative is based on the test statistic defined as the measuring how well the data $(\srx, \sry, \srz)$ fit the distribution $\bar{\mathcal L}(\pry|\prx,\prz)$. For example, if
\begin{equation}
\bar f(\pry|\prx,\prz) = N(\prx^T\beta + \prz^T \gamma, \sigma^2) \text{ for coefficients } \beta \in \mathbb R^d \text{ and } \gamma \in \mathbb R^p,
\label{linear-model}
\end{equation}
then the optimal test rejects for small values of $\|\sry - \srx \beta - \srz \gamma \|^2$. In Section~\ref{sec:knockoffs}, we establish an analogous optimality statement for MX knockoffs as well. Since the optimal test depends on the alternative distribution $\bar{\mathcal L}(\pry|\prx,\prz)$, CRT and MX knockoffs implementations usually employ a machine learning step to search through the composite alternative (not unlike a likelihood ratio test) for a good approximation $\widehat f_{\pry|\prx,\prz}$. These approximate models are then summarized in various ways to define a test statistic $T$. There is no consensus yet on the best test statistic to use, with some authors \cite{CetL16, SetC17, SetS19} using combinations of fitted coefficients $\widehat \beta$ and others  \cite{Tansey2018, Bates2020} using likelihood-based test statistics. The above optimality results align more closely with the latter strategy. Theorem~\ref{prop:crt-optimality} has inspired an extension of the CRT to the sequential setting using a likelihood-based test statistic, accompanied by a similar optimality result~\cite{Grunwald2022}. Likelihood-based test statistics also have the advantage of avoiding ad hoc combination rules for $\widehat \beta \in \mathbb R^d$ when $d > 1$. It remains to be seen whether likelihood-based or coefficient-based test statistics yield greater power in practice, but a thorough empirical comparison is beyond the scope of this work. For now, it suffices to note that, despite its simplicity, this is the first such power optimality result in the CRT literature.

\paragraph*{}
Intuitively, the results of this section suggest that the more successful $\widehat f_{\pry|\prx,\prz}$ is at approximating the true alternative $f_{\pry|\prx,\prz}$, the more powerful the corresponding CRT will be. We make this relationship precise in an asymptotic setting in Section~\ref{sec:asymptotic-power}. We prepare for these results in the next section by exploring an easier-to-analyze asymptotic equivalent to the CRT.

\section{An asymptotic equivalent to the distilled CRT} \label{sec:weakening}

In Section~\ref{sec:power}, we saw how to construct the optimal test against point alternatives specified by $\bar f_{\pry|\prx,\prz}$. In practice, of course we do not have access to this distribution, so we usually estimate it via a statistical machine learning procedure. The goal of this section and the next is to quantitatively assess the power of the CRT as a function of the prediction error of this machine learning procedure. Specifically, we consider  the power of a specific instance of the CRT (the \textit{distilled CRT (dCRT)} \cite{Liu2020})  against a set of semiparametric alternatives (Section~\ref{sec:asymptotic-setup}). We prepare to assess the power of this test by showing its asymptotic equivalence to the simpler-to-analyze \textit{MX(2) $F$-test} (Section~\ref{sec:second-order-approx}), which is of independent interest due to its closed form and weaker assumptions (Section~\ref{sec:mx2-f-test}). We examine the finite-sample Type-I error control of the MX(2) $F$-test in numerical simulations (Section~\ref{sec:simulations}) and put this section's results into perspective (Section~\ref{sec:comparison-to-existing-results-3}) before moving on to stating the desired power results in the next section (Section~\ref{sec:asymptotic-power}).

\subsection{Semiparametric alternatives and the distilled CRT} \label{sec:asymptotic-setup}

First, we define an asymptotic framework within which we will work in Sections~\ref{sec:weakening} and~\ref{sec:asymptotic-power}. Following a triangular array formalization, for each $n = 1, 2, \dots$, we have a joint law $\mathcal L_n$ over $(\prx, \pry, \prz) \in \mathbb R^{d + r + p}$, where $d = \text{dim}(\prx)$ remains fixed, $r = \text{dim}(\pry) = 1$, and $p = \text{dim}(\prz)$ can vary arbitrarily with $n$. For each $n$, we receive $n$ i.i.d. samples $(\srx, \sry, \srz) = \{(\srx_i, \sry_i, \srz_i)\}_{i = 1}^n$ from $\mathcal L_n$. Note that we leave implicit the dependence on $n$ of $(\prx,\pry,\prz)$ and $(X,Y,Z)$ to lighten the notation. In this framework, it will be useful to define the mean and variance functions
\begin{equation}
	\mu_n(\prz) \equiv \mathbb E_{\mathcal L_n}[\prx|\prz]  \text{ and } \Sigma_n(\prz) \equiv \text{Var}_{\mathcal L_n}[\prx|\prz].
	\label{eq:conditional-mean-variance}
\end{equation}

Now, consider a set of semiparametric (partially linear) alternatives $\mathcal L_n(\pry|\prx,\prz)$ such that
\begin{equation}
	\pry = \prx^T \beta_n + g_n(\prz) + \peps; \quad \peps \sim N(0, \sigma^2 ),\ \sigma^2 > 0
	\label{eq:linearity}
\end{equation}
for $\peps \independent (\prx, \pry, \prz)$. Here, $\beta_n \in \mathbb R^d$ is a coefficient vector, $g_n: \mathbb R^p \rightarrow \mathbb R$ a general function, and $\sigma^2 > 0$ the residual variance. Of special interest are local alternatives where $\beta_n = h/\sqrt{n}$ for some $h \in \mathbb R^d$. We emphasize that---in this section and throughout the paper---we use the partially linear model~\eqref{eq:linearity} exclusively as an alternative distribution against which to assess power, rather than an additional assumption required for Type-I error control. By Theorem~\ref{prop:crt-optimality}, the most powerful test against the alternative~\eqref{eq:linearity} is the CRT based on the likelihood statistic 
\begin{equation}
\begin{split}
T_n^{\text{opt}}(\srx, \sry, \srz) &= \prod_{i = 1}^n\frac{1}{(2\pi)^{1/2}}\exp\left(-\frac{1}{2\sigma^2}\left(\sry_i - \srx_i^T \beta_n - g_n(\srz_i)\right)^2\right) \\
&=\prod_{i = 1}^n\frac{1}{(2\pi)^{1/2}}\exp\left(-\frac{1}{2\sigma^2}\left(\sry_i - (\srx_i-\mu_n(\srz_i))^T \beta_n - g_n'(\srz_i)\right)^2\right),
\label{eq:likelihood-ratio}
\end{split}
\end{equation}
where
\begin{equation}
\bar g_n(\prz) \equiv \mathbb E[\pry|\prz] =  \mu_n(\prz)^T\beta_n  + g_n(\prz).
\label{eq:g-n-prime-def}
\end{equation}
Assuming local alternatives $\beta_n = h/\sqrt n$ and taking a logarithm, we obtain
\begin{equation}
	\begin{split}
		\log T_n^{\text{opt}}(\srx, \sry, \srz) &= -\frac n 2 \log(2\pi) - \frac{1}{2\sigma^2}\sum_{i = 1}^n\left(\sry_i - (\srx_i-\mu_n(\srz_i))^T h/\sqrt{n} - \bar g_n(\srz_i)\right)^2 \\
		&\approx   \frac{h^T}{\sigma^2}\frac{1}{\sqrt{n}}\sum_{i = 1}^n(\sry_i - \bar g_n(\srz_i))(\srx_i-\mu_n(\srz_i)) + C,
	\end{split}
\label{eq:optimal-semiparametric}
\end{equation}
where $C$ is a constant that does not depend on $\srx$ and therefore does not change upon resampling. 

Of course, inference based on $T_n^{\text{opt}}$ is infeasible because the function $\bar g_n$ is unknown in practice. Suppose we have learned an estimate $\widehat g_n$ of this function, possibly in-sample. Then, the derivation~\eqref{eq:optimal-semiparametric} motivates us to base inference on the sample covariance between $\prx$ and $\pry$ after adjusting for $\prz$:
\begin{equation}
	\widehat \rho_n(\srx, \sry, \srz) \equiv \frac1n\sum_{i = 1}^n (\sry_{i} - \widehat g_n(\srz_{i}))(\srx_{i} - \mu_n(\srz_i)).
	\label{rho-hat}
\end{equation}
Consider first the case $d = 1$. The CRT rejecting for large values of $|\widehat \rho_n|$ is an instance of the dCRT \cite{Liu2020}. The idea of the dCRT (Algorithm~\ref{alg:dCRT}) is to \textit{distill}---usually via a machine learning regression method---the information from the high-dimensional $\srz \in \mathbb R^{p \times n}$ about $\srx$ and $\sry$ into a low-dimensional summary $D \in \mathbb R^{q \times n}$,  where $q \ll p$. This is accomplished using a \textit{distillation function} $d: (\sry, \srz) \mapsto D$. Then, the CRT is applied using a test statistic of the form $T_n(\srx, \sry, \srz) \equiv T_n^d(\srx, \sry, D) = T_n^d(\srx, \sry, d(\sry, \srz))$. For example, the CRT based on the statistic $\widehat \rho_n$~\eqref{rho-hat} can be expressed as the dCRT with distillation function $d_i(\sry, \srz) = (\widehat g_n(\srz_i), \mu_n(Z_i))$, where $\widehat g_n$ is learned in-sample on $(\sry, \srz)$. 

\begin{center}
	\begin{minipage}{\linewidth}
		\begin{algorithm}[H]
			\nonl  \textbf{Input:} $\{(\srx_i, \sry_i, \srz_i)\}_{i=1}^n$, distribution  $f^*_{\prx|\prz}$, distillation function $d$, test statistic $T_n^d$, number of resamples $B$\\
			Distill information in $\srz$ about $\srx$ and $\sry$ into $D \equiv d(\sry, \srz)$\;
			\For{$b = 1, 2, \dots, B$}{
				Resample $\srxk_i^{(b)} \overset{\text{ind}}\sim f^*_{\prx|\prz = \srz_i},\ i = 1, \dots, n$\;
			}
			Compute $\widehat p \equiv \frac{1}{B+1}\sum_{b = 1}^B \mathbbm 1(T_n^d(\srxk^{(b)}, \sry, D) \geq T_n^d(\srx, \sry, D))$.\\
			\nonl \textbf{Output:} dCRT $p$-value $\widehat p$.\\
			\nonl \textbf{Computational cost:} One $p$-dimensional model fit, and drawing $B$ resamples.
			\caption{\bf The distilled conditional randomization test (dCRT)}
			\label{alg:dCRT}
		\end{algorithm}
	\end{minipage}
\end{center}

The dCRT was proposed for its computational speed: The computationally expensive distillation step is a function only of $(\sry, \srz)$, so it need not be refit upon resampling $\widetilde X$. By contrast, the originally proposed instance of the CRT \cite{CetL16} involved learning $\widehat f_{\pry|\prx, \prz}$ on the entire sample $(\srx, \sry, \srz)$, and therefore the learning procedure needed to be re-applied to each resampled dataset $(\srxk^{(b)}, \sry, \srz)$. The derivations~\eqref{eq:likelihood-ratio} and~\eqref{eq:optimal-semiparametric} suggest that the dCRT is not only computationally fast, but also a natural test to consider for power against semiparametric alternatives~\eqref{eq:linearity}. We therefore focus on this class of tests.

In preparation to study the power of the dCRT, we extend it to $d > 1$ and propose an asymptotically equivalent test that is easier to analyze.


\subsection{A second-order approximation to the dCRT} \label{sec:second-order-approx}

Let us consider first the special case
\begin{equation}
\mathcal L_n(\prx | \prz) = N(\mu_n(\prz), \Sigma_n(\prz)).
\label{eq:normal-conditional}
\end{equation}
In this case, the resampling distribution of $\widehat \rho_n$ can be computed in closed form \cite{Liu2020}:
\begin{equation}
\mathcal L_n(\sqrt n \cdot \widehat \rho_n(\srxk, \sry, \srz) \mid \srx, \sry, \srz) = N(0, \widehat S_n^2),
\label{eq:normal-resampling-distribution}
\end{equation}
where
\begin{equation}
	\widehat S_n^2 \equiv \frac{1}{n}\sum_{i = 1}^n (\sry_{i} - \widehat g_n(\srz_{i}))^2\Sigma_n(\srz_i).
	\label{eq:s-n-hat}
\end{equation}
When $d = 1$, the dCRT based on the statistic $T_n(\srx, \sry, \srz) = |\sqrt n \cdot \widehat \rho_n(\srx, \sry, \srz)|$ (and infinitely many resamples $B$) therefore rejects when $T_n(\srx, \sry, \srz) > \widehat S_n \cdot z_{1-\alpha/2}$, requiring no resampling. To extend this to $d > 1$, consider the standardized quantity
\begin{equation}
U_n(\srx, \sry, \srz) \equiv \widehat S_n^{-1} \sqrt{n} \widehat \rho_n = \frac{\widehat S_n^{-1}}{\sqrt{n}}\sum_{i = 1}^n (\sry_{i} - \widehat g_n(\srz_{i}))(\srx_{i} - \mu_n(\srz_i)) \in \mathbb R^d.
\label{u-hat}
\end{equation}
It is natural to use as a test statistic the squared norm of $U_n$:
\begin{equation}
	T_n(\srx, \sry, \srz) \equiv \|U_n(\srx, \sry, \srz)\|^2.
	\label{t-n}
\end{equation}
Then, the normal resampling distribution~\eqref{eq:normal-resampling-distribution} implies that
\begin{equation}
\mathcal L_n(T_n(\srxk, \sry, \srz)|\srx, \sry, \srz) = \chi^2_d.
\end{equation}
It follows that the dCRT based on test statistic $T_n(\srx, \sry, \srz)$ yields the test
\begin{equation}
\phi^{N(\mu_n, \Sigma_n)}_n(\srx, \sry, \srz) \equiv \mathbbm 1(T_n(\srx, \sry, \srz) > c_{d,1-\alpha}),
\label{eq:crt-closed-form}
\end{equation}
where we recall that $c_{d,1-\alpha}$ is defined as the $1-\alpha$ quantile of $\chi^2_d$. Note that all tests $\phi$ in Sections~\ref{sec:weakening} and~\ref{sec:asymptotic-power} will be (d)CRTs based on the test statistic $T_n$~\eqref{t-n}. To ease notation, we therefore omit the subscript $T_n$ and the superscript ``CRT'' from the notation introduced in equation~\eqref{eq:randomized-CRT}, replacing these with $n$ and the distribution of $\prx|\prz$ with respect to which resampling is done, respectively. For example, the superscript in the test defined in equation~\eqref{eq:crt-closed-form} is based on the resampling distribution $\prx|\prz \sim N(\mu_n(\prz), \Sigma_n(\prz))$.

If the conditional distribution $\mathcal L_n(\prx|\prz)$ is not Gaussian, then the dCRT $\phi_n^{\mathcal L_n}$ based on $T_n(\srx, \sry, \srz)$ will not reduce to the closed-form expression~\eqref{eq:crt-closed-form}. However, we can think of the test $\phi^{N(\mu_n, \Sigma_n)}_n$ as a kind of second-order approximation for $\phi_n^{\mathcal L_n}$ as long as $\mathcal L_n(\prx|\prz)$ has first and second moments given by $\mu_n(\prz)$ and $\Sigma_n(\prz)$, respectively. Indeed, it is easy to check that the resampling distribution $\mathcal L_n(\sqrt n \cdot \widehat \rho_n(\srxk, \sry, \srz) \mid \srx, \sry, \srz)$ matches that derived in the normal case~\eqref{eq:normal-resampling-distribution} up to two moments. Under a few assumptions, we can make this intuition precise by showing that $\phi_n^{\mathcal L_n}$ is asymptotically equivalent to $\phi^{N(\mu_n, \Sigma_n)}_n$ (Theorem~\ref{thm:equivalence} below). We require the distribution $\mathcal L_n$ to satisfy the following moment conditions \footnote{The exponents in these moment conditions can be relaxed from 4 to $2+\delta$, in particular, requiring an appropriate triangular array weak law of large numbers with 1+$\delta$ moments. This slight weakening of moment conditions requires significantly more technical effort, so is omitted for simplicity since it does not alter the main takeaway messages of our analysis.}for fixed $c_1, c_2 > 0$:
\small
\begin{equation}
\mathcal L_n \in \mathscr L_n(c_1, c_2) \equiv \{\mathcal L_n: \|S_{n}^{-1}\| \leq c_1, \mathbb E_{\mathcal L_n}\left[(\pry - \widehat g_n(\prz))^{4} \mathbb E_{\mathcal L_n}[\|\prx - \mu_n(\prz)\|^{4}|\prz]\right] \leq c_2 \},
\label{eq:moment-conditions}
\end{equation} 
\normalsize
where
\begin{equation}
S_n^2 \equiv \mathbb E[\widehat S_n^2] = \mathbb E_{\mathcal L_n}\left[(\pry - \widehat g_n(\prz))^2 \Sigma_n(\prz)\right].
\end{equation}
Furthermore, to avoid technical complications, we assume that the estimate $\widehat g_n$ is trained on an independent dataset (whose size can vary arbitrarily with $n$ and is not included in the sample size $n$ used for testing). For example, there has been recent interest in combining observational and experimental (randomized) data; typically, the former is much more abundant than the latter. We can think of $\widehat g_n$ being trained on the former, and then used for MX inference on the latter \cite{Bates2020}.
These training sets across $n$ and resulting estimates $\widehat g_n$ remain fixed throughout.

\begin{theorem} \label{thm:equivalence}
Suppose that for each $n$, $\mathcal L_n$ is a law whose first and second conditional moments are given by $\mu_n(\prz)$ and $\Sigma_n(\prz)$~\eqref{eq:conditional-mean-variance}, which satisfies the moment conditions~\eqref{eq:moment-conditions} for fixed for some $c_1, c_2 > 0$. Let $\phi_n^{\mathcal L_n}$ be the dCRT based on the test statistic $T_n(\srx, \sry, \srz)$~\eqref{eq:s-n-hat}, \eqref{u-hat}, \eqref{t-n}, with $\widehat g_n$ trained out of sample. The threshold $C_n(\sry, \srz)$ of this test~\eqref{upper-quantile} converges in probability to the $\chi^2_d$ quantile:
	\begin{equation}
		C_n(Y,Z) \overset{\mathcal L_n}\rightarrow_p c_{d,1-\alpha}.
		\label{eq:threshold-convergence}
	\end{equation}
	Furthermore, if $T_n(\srx, \sry, \srz)$ does not accumulate near $c_{d,1-\alpha}$, i.e.
	\begin{equation}
		\lim_{\delta \rightarrow 0}\limsup_{n \rightarrow \infty}\ \mathbb P_{\mathcal L_n}[|T_n(\srx, \sry, \srz)-c_{d,1-\alpha}| \leq \delta] = 0,
		\label{eq:non-accumulation}
	\end{equation}
	then the dCRT $\phi_n^{\mathcal L_n}$ is asymptotically equivalent to its second order approximation $\phi^{N(\mu_n, \Sigma_n)}_n$~\eqref{eq:crt-closed-form}:
	\begin{equation}
		\lim_{n \rightarrow \infty}\ \mathbb P_{\mathcal L_n}[\phi^{\mathcal L_n}_n(\srx, \sry, \srz) \neq \phi^{N(\mu_n, \Sigma_n)}_n(\srx, \sry, \srz)] = 0.
		\label{eq:asymptotic-equivalence}
	\end{equation}	
\end{theorem}

Informally, this theorem (proved in Appendix~\ref{sec:proofs-sec45}) suggests that the CRT resampling distribution of $T_n(\srx, \sry, \srz)$ converges to $\chi^2_d$, which is the resampling distribution of this test statistic under a normal $\mathcal L_n(\prx|\prz)$. Note that the resulting equivalence~\eqref{eq:asymptotic-equivalence} holds for the specific instance of the CRT based on the statistic $T_n$ defined in via equations~\eqref{u-hat} and~\eqref{t-n}, though other kinds of test statistics may lead to similar large-sample behavior. While Theorem~\ref{thm:equivalence} is stated for $\widehat g_n$ trained out of sample, we conjecture that it continues to hold when $\widehat g_n$ is fit in sample, as in the original dCRT construction \cite{Liu2020}. At least, we observe that the conditioning in the construction of the resampling distribution $\mathcal L_n(\sqrt n \cdot \widehat \rho_n(\srxk, \sry, \srz) \mid \srx, \sry, \srz)$ ensures that its mean and variance remain equal to $0$ and $\widehat S_n^2$ even when $\widehat g_n$ is fit in sample.

Theorem~\ref{thm:equivalence} has several consequences. First, it allows us to study the power of the dCRT $\phi_n^{\mathcal L_n}$ against semiparametric alternatives~\eqref{eq:linearity} by studying instead the simpler test $\phi_n^{N(\mu_n, \Sigma_n)}$. We pursue this direction in Section~\ref{sec:asymptotic-power}. Second, it implies a certain robustness property of the dCRT. Indeed, suppose we run the dCRT based on an incorrect law $\mathcal L'_n \neq \mathcal L_n$, but whose first and second moments match that of $\mathcal L_n$ and such that $\mathcal L_n$ is contiguous with respect to $\mathcal L'_n$. Then, applying Theorem~\ref{thm:equivalence} to $\mathcal L_n$ and $\mathcal L'_n$ implies that $\mathbb P_{\mathcal L_n}[\phi^{\mathcal L'_n}_n(\srx, \sry, \srz) \neq \phi^{\mathcal L_n}_n(\srx, \sry, \srz)] \rightarrow 0$. It follows that since $\phi^{\mathcal L_n}_n$ controls the type-I error asymptotically (in fact, also in finite samples), then so does $\phi^{\mathcal L'_n}_n$. We omit the formal statement of this result for the sake of brevity. Third, it suggests a distinct conditional independence test with valid Type-I error control under the weaker assumption that only the first two moments of $\mathcal L_n(\prx|\prz)$ are known. We expand on this third consequence next.

\subsection{The MX(2) assumption and the MX(2) $F$-test} \label{sec:mx2-f-test}

The asymptotic equivalence of $\phi_n^{N(\mu_n, \Sigma_n)}$ to $\phi_n^{\mathcal L_n}$ stated in Theorem~\ref{thm:equivalence} suggests that we may replace the dCRT based on the law $\mathcal L_n(\prx|\prz)$ with that based on its normal approximation $N(\mu_n(\prz), \Sigma_n(\prz))$ while preserving Type-I error control. Since the test $\phi_n^{N(\mu_n, \Sigma_n)}$ requires knowledge only of the first two moments $\mu_n(\prx)$ and $\Sigma_n(\prz)$, this means that we may control Type-I error without the full MX assumption. To formalize this, let us define the 
\begin{equation}
	\begin{split}
&\textit{MX(2) assumption: } \text{the conditional mean } \mu_n(\prz) \equiv \mathbb E_{\mathcal L_n}[\prx|\prz] \\ 
&\text{ and conditional variance } \Sigma_n(\prz) \equiv \text{Var}_{\mathcal L_n}[\prx|\prz] \text{ are known}.
	\label{eq:mx-2-assumption}
	\end{split}
\end{equation}
By analogy with definition~\eqref{eq:null-under-modelX}, the MX(2) null hypothesis is defined as 
\begin{equation}
	\mathscr L_0^{\text{MX(2)}} = \mathscr L_0^{\text{MX(2)}}(\mu_n(\cdot), \Sigma_n(\cdot)) \equiv \mathscr L_{0} \cap \mathscr L^{\text{MX(2)}}(\mu_n(\cdot), \Sigma_n(\cdot)),
	\label{mx2-null}
\end{equation}
where
\begin{equation*}
	\begin{split}
		\mathscr L^{\textnormal{MX(2)}}(\mu_n(\cdot), \Sigma_n(\cdot)) \equiv \{\mathcal L_n: \mathbb E_{\mathcal L_n}[\prx|\prz] = \mu_n(\prz),\ \text{Var}_{\mathcal L_n}[\prx|\prz] = \Sigma_n(\prz)\}.
	\end{split}
\end{equation*}

Under the MX(2) assumption, the CRT is undefined because there is no conditional distribution $\mathcal L_n(\prx|\prz)$ to resample from. Nevertheless, we may define the \textit{MX(2) $F$-test} by running the resampling-free dCRT as though $\mathcal L_n(\prx|\prz)$ were normal, with the given first and second moments (Algorithm~\ref{alg:MX(2)-F-test}). We denote this test $\phi_n^{N(\mu_n, \Sigma_n)}$, as before.
\begin{center}
	\begin{minipage}{\linewidth}
		\begin{algorithm}[H]
			\KwData{
				$\{(\srx_i, \sry_i, \srz_i)\}_{i=1}^n$, $\mu_n(\cdot)$ and $\Sigma_n(\cdot)$ in \eqref{eq:conditional-mean-variance}, learning method $g$
			}
			Obtain $\widehat g_n$ by fitting $g$ out of sample\;
			Recall $\mu_n(\srz_i) \equiv \mathbb E_{\mathcal L_n}[\srx_i|\srz_i], \text{ set } \widehat S_n^2 \equiv \frac{1}{n}\sum_{i = 1}^n (\sry_{i} - \widehat g_n(\srz_{i}))^2\Sigma_n(\srz_i)$\;
			Set $\smash{U_n \equiv \frac{\widehat S_n^{-1}}{\sqrt{n}}\sum_{i = 1}^n (\sry_{i} - \widehat g_n(\srz_{i}))(\srx_{i} - \mu_n(\srz_i))}$ and $T_n = \|U_n\|^2$\;
			Compute $\widehat p^{\text{MX(2)}} \equiv \mathbb P[\chi^2_d > T_n]$.\\
			\nonl \textbf{Output:} MX(2) $F$-test asymptotic $p$-value $\widehat p^{\text{MX(2)}}$.\\
			\nonl \textbf{Computational cost:} One $p$-dimensional model fit.
			\caption{\bf The MX(2) $\bm F$-test}
			\label{alg:MX(2)-F-test}
		\end{algorithm}
	\end{minipage}
\end{center}
Note that a one-sided version of this test (the \textit{MX(2) $t$-test}) can be defined for $d = 1$ by rejecting for large values of $U_n(\srx, \sry, \srz)$. 

The MX(2) $F$-test controls the Type-I error under the MX(2) assumption, if the moment conditions~\eqref{eq:moment-conditions} hold and $\widehat g_n$ is fit out of sample.
\begin{theorem} \label{thm:asymptotic-alpha-level}
If $\mathcal L_n \in \mathscr L_0^{\textnormal{MX(2)}}\cap \mathscr L_n(c_1, c_2)$ for some $c_1, c_2 > 0$ and $\widehat g_n$ is fit out of sample, then the standardized quantity $U_n(\srx, \sry, \srz)$ converges to the standard normal: 
\begin{equation}
	U_n(\srx, \sry, \srz) \overset{\mathcal L_n}\rightarrow_d N(0, I_d).
	\label{eq:asymptotic-normality}
\end{equation}
Therefore, the MX(2) $F$-test controls Type-I error asympotically, uniformly over the above subset of $\mathscr L_0^{\textnormal{MX(2)}}$: 
\begin{equation}
	\limsup_{n \rightarrow \infty} \sup_{\mathcal L_n \in \mathscr L_0^{\textnormal{MX(2)}}\cap \mathscr L_n(c_1, c_2)} \mathbb E_{\mathcal L_n}[\phi^{N(\mu_n, \Sigma_n)}_n(\srx, \sry, \srz)] \leq \alpha.
	\label{asymptotic-alpha-level}
\end{equation}
\end{theorem}

See Appendix~\ref{sec:proofs-sec45} for a proof of this theorem. The moment assumptions can be relaxed for pointwise error control (less desirable), but are unavoidable for uniform type-I error control as stated in the corollary. More importantly, we conjecture that the MX(2) $F$-test continues to have asymptotic Type-I error control even if $\widehat g_n$ is fit in sample. One may expect this because the validity of the MX(2) $F$-test derives from the correctness of $(\mu_n, \Sigma_n)$ rather than that of $\widehat g_n$. This conjecture is supported by the results of a simulation study presented in Appendix~\ref{sec:simulations}.

\subsection{Comparison to existing results} \label{sec:comparison-to-existing-results-3}

\paragraph{Comparison to model-X literature.}

The preceding results suggest that the MX(2) $F$-test is a useful alternative to the dCRT: the power of these methods is asymptotically the same (Theorem~\ref{thm:equivalence}), while the MX(2) $F$-test is computationally faster because it does not require resampling (Table~\ref{tab:summary-1}). On the other hand, note that we have proven Type-I error control for the MX(2) $F$-test only when $\widehat g_n$ is fit out of sample and only asymptotically, while the dCRT gives finite-sample Type-I error control with in-sample fit $\widehat g_n$ (albeit under the stronger model-X assumption). However, numerical simulations suggest good finite-sample Type-I error control for the MX(2) $F$-test even when $\widehat g_n$ is fit in sample. Furthermore, Theorem~\ref{thm:asymptotic-alpha-level} shows that asymptotic Type-I error control of MX-style methodologies can be achieved under the weaker MX(2) assumption~\eqref{eq:mx-2-assumption}, requiring only two moments of the conditional distribution $\mathcal L_n(\prx|\prz)$ rather than the entire conditional distribution (Table~\ref{tab:summary-2}). If strict Type-I error control in finite samples is desired, however, then we must continue to rely on the full MX assumption. 
Finally, note that the MX(2) assumption still requires \textit{exact} knowledge of the first and second conditional moments; we leave as an important future direction to examine the robustness of these tests to errors in these quantities. First steps in this direction have been taken recently \cite{Berrett2019, Li2022a}.



\paragraph{Comparison to doubly-robust literature.}

The semiparametric model~\eqref{eq:linearity} has been extensively studied (see e.g. the classic works \cite{Robinson1988, Robins1992}), in which context estimation of the parameter $\beta_n$ is well understood. By contrast, we do not assume the validity of the semiparametric model, using it only as an alternative against which to evaluate power. A related and perhaps more relevant line of work is non-parametric doubly robust testing \cite{Shah2018, Dukes2020}  and estimation \cite{VanderLaan2011, Chernozhukov2018}. Here, the inferential target is some functional of the data-generating distribution. The most relevant such functional is the expected conditional covariance $\rho_n \equiv \mathbb E_{\mathcal L_n}[\text{Cov}_{\mathcal L_n}[\prx, \pry|\prz]]$. Note that a valid test of the null hypothesis $H_0: \rho_n = 0$ is also a valid test of the conditional independence hypothesis $H_0: \prx \independent \pry \mid \prz$, since conditional independence implies that $\rho_n = 0$ (though the converse is not true in general). The quantity $\widehat \rho_n$ turns out to be a consistent estimator of $\rho_n$ under the MX(2) assumption (Lemma~\ref{lem:consistency}). Such product-of-residuals estimators are also commonly employed in the semi- and non-parametric literatures \cite{Robinson1988, Robins1992, Li2011}.

To compare our results with those in non-parametric doubly robust inference, we consider the closest representative of the latter: the generalized covariance measure (GCM) test of Shah and Peters \cite{Shah2018}. For $d = 1$, the GCM test statistic is defined as
\begin{equation}
	\widehat \rho^{\text{GCM}}_n \equiv \frac{1}{n}\sum_{i = 1}^n  (\sry_i - \widehat g_n(\srz_i))(\srx_i - \widehat \mu_n(\srz_i)),
\end{equation}
where $\widehat \mu_n(\prz)$ and $\widehat g_n(\prz)$ are estimates of $\mu_n(\prz) = \mathbb E_{\mathcal L_n}[\prx|\prz]$ and $g_n(\prz) \equiv \mathbb E_{\mathcal L_n}[\pry|\prz]$, respectively. This statistic is shown to converge under conditional independence to a mean-zero normal limit as long as the estimates of $\mathbb E_{\mathcal L_n}[\prx|\prz]$ and $\mathbb E_{\mathcal L_n}[\pry|\prz]$ are both consistent, while the product in these estimation errors tends to zero at a rate of $o(n^{-1/2})$. By contrast, the MX(2) $F$-test places more weight on the model for $\prx|\prz$ (assuming both first and second moments of this conditional distribution are known) while placing less weight on the model for $\pry|\prz$ (not assuming even consistency for $\mathbb E_{\mathcal L_n}[\pry|\prz]$). Therefore, while the MX(2) $F$-test closely resembles the GCM test, the assumptions required for validity of these two methods do not subsume each other (Table~\ref{tab:summary-2}). 


	\begin{table}[h!]
	\centering
	\begin{tabular}{l|ll}
		Method & Guarantee & Resampling \\
		\hline
		CRT & Finite-sample & Yes \\	
		MX(2) $F$-test & Asymptotic & No  \\
		GCM test & Asymptotic & No
	\end{tabular}
	\caption{Type-I error guarantee and necessity of resampling for each method compared.}
	\label{tab:summary-1}
\end{table}

\begin{table}[h!]
	\small
	\begin{tabular}{l|lllll}
		Method & $\mathcal E(\mathbb E[\prx|\prz]$) & $\mathcal E(\text{Var}[\prx|\prz])$ & $\mathcal E(\mathcal L(\prx|\prz))$ & $\mathcal E(\mathbb E[\pry|\prz]$) & $\mathcal E(\mathbb E[\prx|\prz]) \times \mathcal E(\mathbb E[\pry|\prz])$ \\
		\hline
		CRT & 0 & 0 & 0 & --& --  \\
		MX(2) $F$-test  & 0 & 0 & -- & -- & -- \\
		GCM test & $o_p(1)$ & -- & -- & $o_p(1)$ & $o_p(n^{-1/2})$
	\end{tabular}
	\caption{Assumptions necessary for each method compared (excluding moment assumptions). Here, $\mathcal E(\cdot)$ refers to the root-mean-squared estimation error of a given quantity.}
	\label{tab:summary-2}
\end{table}

\paragraph{Comparison to causal inference literature.}

Theorem~\ref{thm:equivalence} is a statement about the asymptotic equivalence between the resampling-based CRT and the asymptotic MX(2) $F$-test. The MX CRT is in the spirit of the finite-population approach to causal inference (Fisher), whereas the MX(2) $F$-test is in the spirit of the asymptotic super-population approach (Neyman). We find that research in these two strands of work on causal inference have proceeded largely separately from each other, and therefore connections between the two have received relatively little attention. However, there has been a recent line of work \cite{Ding2017,Wu2020a,Zhao2021} focusing on the asymptotic behavior of the Fisher randomization test in the context of completely randomized experiments. A similar result to Theorem~\ref{thm:equivalence} is that the Fisher randomization test (analogous to the CRT) is asymptotically equivalent to the Rao score test (analogous to the MX(2) $F$-test) in a completely randomized experiment \cite[Theorem A.1]{Ding2017}. Theorem~\ref{thm:equivalence} can be viewed as an extension of this result to accommodate for non-binary treatments as well as high-dimensional covariates affecting both treatment and response.


\paragraph{}

Having found that the dCRT is a natural test to apply for power against semiparametric alternatives, and that this test is equivalent to the simpler MX(2) $F$-test, we are ready to study the relationship between the power of the dCRT and the quality of the underlying machine learning procedure.

%

\section{dCRT power against semiparametric alternatives} \label{sec:asymptotic-power}

In this section, we present our results on the asymptotic power of the dCRT against the semiparametric alternatives~\eqref{eq:linearity}. We state these results first (Section~\ref{sec:power-results}), then apply these to lasso-based dCRT (Section~\ref{sec:power-lasso-based}), and finally compare our results to existing ones (Section~\ref{sec:comparison-to-existing-results-4}). All proofs are deferred to Appendix~\ref{sec:proofs-sec45}.

\subsection{Power against semiparametric alternatives} \label{sec:power-results}

In Theorem~\ref{thm:power} below, we express the asymptotic power of the dCRT against alternatives~\eqref{eq:linearity} in terms of the variance-weighted mean square  error of $\widehat g_n$:
\begin{equation} 
	\mathcal E^2_n \equiv  \mathbb E_{\mathcal L_n}\left[(\widehat g_n(\prz)-\bar g_n(\prz))^2 \cdot \overline \Sigma_n^{-1/2}\Sigma_n(\prz)\overline \Sigma_n^{-1/2}\right], \quad \text{where} \ \overline \Sigma_n\equiv \mathbb E_{\mathcal L_n}[\Sigma_n(\prz)].
\end{equation}
Recall from definition~\eqref{eq:g-n-prime-def} that $\bar g_n(\prz) \equiv \mathbb E[\pry|\prz]$. Note that if $(\prx, \prz)$ is jointly Gaussian, then $\Sigma_n(\prz) = \overline \Sigma_n$ for all $\prz$ and therefore $\mathcal E_n^2 = \mathbb E_{\mathcal L_n}[(\widehat g_n(\prz)-\bar g_n(\prz))^2] \cdot I_d$. Our result requires the following moment assumptions:
	\begin{equation}
\sup_{n} \|\overline \Sigma_n^{-1}\| < \infty,
\label{eq:s-n-inverse-assump}
\end{equation}
\begin{equation}
	\sup_n\ \mathbb E_{\mathcal L_n}[\|\prx - \mu_n(\prz)\|^8] < \infty,
	\label{eq:eighth-moment-assump-1}
\end{equation}
and
\begin{equation}
	\sup_n\ \mathbb E_{\mathcal L_n}[(\widehat g_n(\prz)-\bar g_n(\prz))^4\|\prx - \mu_n(\prz)\|^4] < \infty.
	\label{eq:eighth-moment-assump-2}
\end{equation}

\begin{theorem}  \label{thm:power}
Suppose $\mathcal L_n$ and $\widehat g_n$ (trained out of sample) are such that the conditional distribution $\mathcal L_n(\pry|\prx, \prz)$ follows the semiparametric alternative~\eqref{eq:linearity}, the moment conditions~\eqref{eq:s-n-inverse-assump}, \eqref{eq:eighth-moment-assump-1}, and \eqref{eq:eighth-moment-assump-2} are satisfied, and that the conditional variance and variance-weighted mean squared error converge:
\begin{equation}
	\overline \Sigma_n \rightarrow \overline{\Sigma} \quad \text{and} \quad \mathcal E_n^2 \rightarrow \mathcal E^2 \quad \text{as } n \rightarrow \infty.
	\label{eq:limits}
\end{equation}
Then, we have the following two statements:
\begin{enumerate}
\item[(a)] (Consistency) If $\beta_n = \beta \neq 0$ for each $n$, then the dCRT $\phi_n^{\mathcal L_n}$ and the MX(2) $F$-test $\phi_n^{N(\mu_n, \Sigma_n)}$ are consistent:
\begin{equation}
	\begin{split}
		\lim_{n \rightarrow \infty} \mathbb E_{\mathcal L_n}\left[\phi_n^{\mathcal L_n}(\srx, \sry, \srz)\right] &= \lim_{n \rightarrow \infty} \mathbb E_{\mathcal L_n}\left[\phi_n^{N(\mu_n, \Sigma_n)}(\srx, \sry, \srz)\right] = 1.
		\label{eq:consistency}
	\end{split}
\end{equation}
\item[(b)] (Power against local alternatives) If $\beta_n = h_n/\sqrt{n}$ for a convergent sequence ${h_n \rightarrow h \in \mathbb R^d}$, then
\begin{equation}
	\begin{split}
			\lim_{n \rightarrow \infty} \mathbb E_{\mathcal L_n}\left[\phi_n^{\mathcal L_n}(\srx, \sry, \srz)\right] &= \lim_{n \rightarrow \infty} \mathbb E_{\mathcal L_n}\left[\phi_n^{N(\mu_n, \Sigma_n)}(\srx, \sry, \srz)\right] \\
		&= \mathbb P[\chi^2_d(\|(\sigma^2I_d +\mathcal E^2)^{-1/2}\overline \Sigma^{1/2} h\|^2) > c_{d,1-\alpha}].
		\label{eq:main-conclusion}
	\end{split}
\end{equation}
\end{enumerate}
\end{theorem}

This theorem is proved in Appendix~\ref{sec:proofs-sec45}. Recalling that $\chi^2_d(\lambda)$ denotes the noncentral chi-square distribution with $d$ degrees of freedom and non-centrality parameter $\lambda$ and $c_{d,1-\alpha}$ denotes the $1-\alpha$ quantile of $\chi^2_d$, the second part of Theorem~\ref{thm:power} states that the dCRT has power equal to that of a $\chi^2$ test of a multivariate normal random vector having mean zero under the alternative $N((\sigma^2I_d +\mathcal E^2)^{-1/2}\overline \Sigma^{1/2} h, I_d)$. This result establishes a direct link between the estimation error in $\widehat g_n$ and the power of the CRT against local alternatives. In particular, the mean-squared error term $\mathcal E^2$ contributes additively to the irreducible error term $\sigma^2 I_d$. We can gain intuition for this result by considering the regression model
\begin{equation}
\begin{split}
\pry - \widehat g_n(\prz) &= (\prx - \mu_n(\prz))^T\beta_n + (\bar g_n(\prz) - \widehat g_n(\prz) + \bm \eps)
\label{eq:regression-model}
\end{split}
\end{equation}
obtained from the semiparametric model~\eqref{eq:linearity} by subtracting $\widehat g_n(\prz)$ from both sides. The test statistic $T_n$ is based on the quantity $\widehat \rho_n$ defined in equation~\eqref{rho-hat}, which can be viewed as an unnormalized version of the fitted regression coefficients of $\sry - \widehat g_n(\srz)$ on $\srx - \mu_n(\srz)$. The term $\bar g_n(\prz) - \widehat g_n(\prz)$ in the regression model~\eqref{eq:regression-model} contributes additively to the residual error term, so in a traditional regression analysis we would expect the power of the test to depend on the variance of this error term. In fact, standard large-sample OLS theory (see e.g. Section 2.3 of Hayashi's book \cite{Hayashi2000}) states that the power against local alternatives of the $F$-test in the regression model~\eqref{eq:regression-model} is exactly the same as that of the dCRT (and MX(2) $F$-test) stated in equation~\eqref{eq:main-conclusion}. Of course, the usual $F$-test applied to the regression~\eqref{eq:regression-model} relies on the validity of this model while the dCRT and MX(2) $F$-test instead rely on knowledge of $\text{Var}[\prx|\prz]$. Note that \cite{Wang2020b} also find the power of an MX test and a classical OLS test to have the same power (see their Appendix F).

\subsection{Example: Power of lasso-based CRT} \label{sec:power-lasso-based}

A key ingredient in the power formula~\eqref{eq:main-conclusion} is the limiting variance-weighted mean squared error $\mathcal E^2$. This error depends on the machine learning method used to obtain $\widehat g_n$. We can leverage existing results about the asymptotic behavior of prediction error of machine learning methods in high dimensions. In this section, we consider the case when $\widehat g_n$ is trained using the lasso in the orthogonal design case, which was studied by Bayati and Montanari \cite{Bayati2011}. Note that a recent extension of Bayati and Montanari's results to correlated designs \cite{Celentano2020} can also be used in tandem with~\eqref{eq:main-conclusion}, but we focus our exposition on the orthogonal design case for the sake of simplicity.

\begin{setting}[\bf Linear regression with orthogonal design] \label{setting:orthogonal-design}
Consider a sequence of laws $\mathcal L_n$ such that $\mathcal L_n(\prx, \prz) = N(0, I_{1+p})$ and such that $\mathcal L_n(\pry|\prx, \prz)$ follows the semiparametric model~\eqref{eq:linearity},  with $\beta_n = h_n/\sqrt{n}$ for some convergent sequence $h_n \rightarrow h \in \mathbb R$ and $g_n(\prz) = \prz^T \gamma_n$ for a sequence $\gamma_n \in \mathbb R^p$ such that the entries of $\sqrt n \gamma_n$ converge weakly to a random variable $\Gamma$ on $\mathbb R$ with $\mathbb P[\Gamma \neq 0] > 0$ and $\|\sqrt n \gamma_n\|^2/p \rightarrow \mathbb E[\Gamma^2] < \infty$.
\end{setting}

Until now, we have denoted by $n$ the sample size used for constructing tests, leaving unspecified the size of the separate sample used to train $\widehat g_n$. To get concrete expressions for the power of the dCRT based on a specific machine learning method to obtain $\widehat g_n$, we must take the training sample size into account, which we will do via sample splitting for convenience. We therefore define the test $\varphi_n^{\mathcal L_n}(\srx, \sry, \srz)$, which for some training proportion $\pi \in (0,1)$ split the data into $\pi n$ training observations $(\srx_{\text{train}}, \sry_{\text{train}}, \srz_{\text{train}})$ and $(1-\pi)n$ test observations $(\srx_{\text{test}}, \sry_{\text{test}}, \srz_{\text{test}})$. This test proceeds by first running a lasso of $\sry_{\text{train}}$ on $\srz_{\text{train}}$ with regularization parameter $\lambda$ to obtain an estimate $\widehat \gamma_{\pi n}$. The test $\varphi_n^{\mathcal L_n}(\srx, \sry, \srz)$ is then obtained by running the dCRT on the test data $(\srx_{\text{test}}, \sry_{\text{test}}, \srz_{\text{test}})$ using the estimate $\widehat g_n(\prz) = \prz^T \widehat \gamma_{\pi n}$:
\begin{equation*}
	\varphi_n^{\mathcal L_n}(\srx, \sry, \srz) \equiv \phi_{(1-\pi)n}^{\mathcal L_n}(\srx_{\text{test}}, \sry_{\text{test}}, \srz_{\text{test}}).
\end{equation*}
Note that the dependence of $\phi_{(1-\pi)n}^{\mathcal L_n}(\srx_{\text{test}}, \sry_{\text{test}}, \srz_{\text{test}})$ on the training data $(\srx_{\text{train}}, \sry_{\text{train}}, \srz_{\text{train}})$ is left implicit. 

Under Setting~\ref{setting:orthogonal-design}, we can directly use Bayati and Montanari's theory \cite{Bayati2011} to obtain 
\begin{equation}
\lim_{n \rightarrow \infty}\mathcal E_n^2 = \tau_*^2 - \sigma^2 \quad \text{a.s. in } (\srx_{\text{train}}, \sry_{\text{train}}, \srz_{\text{train}}),
\label{eq:bm-result}
\end{equation}
where $(\alpha_*,\tau_*)$ is the unique solution of the system below:
\begin{equation}
	\begin{split}
		\lambda &= \alpha \tau (1-(\pi\delta)^{-1}\mathbb E[\eta'(\sqrt \pi \Gamma + \tau W; \alpha \tau)]), \\
		\tau^2 &= \sigma^2 + (\pi\delta)^{-1}\mathbb E[(\eta(\sqrt \pi\Gamma + \tau W; \alpha \tau) - \sqrt \pi\Gamma)^2].
	\end{split}
	\label{eq:amp-system}
\end{equation}
Here, $W \sim N(0,1)$ is independent of $\Gamma$ and $\eta(x; \theta) = (|x|-\theta)_+\textnormal{sign}(x)$ is the soft threshold function. This leads to the following corollary of Theorem~\ref{thm:power}, proved in Appendix~\ref{sec:proofs-sec45}:
\begin{corollary} \label{cor:lasso}
Under Setting~\ref{setting:orthogonal-design}, the asymptotic power of the dCRT converges to that of a standard normal location test with alternative mean $\tau_*^{-1} h\sqrt{1-\pi}$:
\begin{equation}
	\begin{split}
	\lim_{n \rightarrow \infty}\mathbb E_{\mathcal L_n}[\varphi^{\mathcal L_n}_{n}(\srx, \sry, \srz)] =  \mathbb P[|N(\tau_*^{-1} h\sqrt{1-\pi},1)| > z_{1-\alpha/2}].
	\label{eq:lasso-power}
	\end{split}
\end{equation}
\end{corollary}
Corollary~\ref{cor:lasso} gives the power of these lasso-based methods in a very simple form, with the prediction error of the lasso entering through the effective noise level $\tau_*$. The impact of the splitting proportion $\pi$ on power can be seen in the multiplication of the signal strength $h$ by $\sqrt{1-\pi}$. The splitting proportion implicitly impacts the effective noise level $\tau_*$ as well; smaller $\pi$ lead to greater effective noise levels. Note that the expectations in Corollary~\ref{cor:lasso} are over both training and test sets, while the expectations in Theorem~\ref{thm:power} are over the test set only. 

\subsection{Comparison to existing results} \label{sec:comparison-to-existing-results-4}

Two other power analyses of the CRT have been recently conducted \cite{Wang2020b, Celentano2020} in parallel to the first version of our paper \cite{Katsevich2020a}, focusing on the case where $g_n(\prz) = \prz^T \gamma_n$, $\widehat g_n$ is trained using the lasso, $n/p \rightarrow \delta$, and the generalized covariance measure test statistic $\widehat \rho_n$ is used. The former study considers the case of orthogonal design (Setting~\ref{setting:orthogonal-design}), while the latter considers arbitrary joint Gaussian distribution for $(\prx,\prz)$. Assuming $\mathbb E_{\mathcal L_n}[\text{Var}_{\mathcal L_n}[\prx|\prz]] \rightarrow s^2$ (the quantity we called $\overline \Sigma$ in Section~\ref{sec:power-results}, with different notation to clarify that for $\text{dim}(\prx) = 1$ the covariance matrix simply becomes a variance), the works \cite{Wang2020b, Celentano2020} found that the power of the CRT with in-sample lasso fit tends to that of a normal location test with alternative mean $sh/\tau_*$, where $\tau_*$ is the effective noise level from AMP theory (in the orthogonal design case, $(\alpha_*, \tau_*)$ are defined by equation~\eqref{eq:amp-system} with $\pi = 1$) and $h$ is the limiting constant of the local alternatives in Setting~\ref{setting:orthogonal-design}. 

This is a similar expression to what we found in Corollary~\ref{cor:lasso} in the orthogonal design case. Furthermore, note that $\tau_*^2 = \sigma^2 + \mathcal E^2$ (i.e. the out-of-sample prediction error of the lasso). It follows that the power expression found by \cite{Wang2020b, Celentano2020} is exactly the same as what we found in part (b) of Theorem~\ref{thm:power}, despite the fact that their $\widehat g_n$ is fit in-sample. 
\cite{Wang2020b} also derive a power expression for the CRT when $\widehat g_n$ is fit in-sample via ordinary least squares (allowing correlated covariates, as we do in Theorem~\ref{thm:power}), which also happens to coincide with expression~\eqref{eq:main-conclusion}. Such in-sample results have been obtained only for these two test statistics, though we conjecture that such results hold more broadly.
By contrast, training $\widehat g_n$ on a separate sample allows us to prove Theorem~\ref{thm:power} for very broad (almost unrestricted) classes of machine learning methods $\widehat g_n$.

Finally, we note a connection between Theorem~\ref{thm:power} and causal inference. It is widely known in causal inference (see e.g. \cite[Section 7.5]{Imbens2015}) that adjustment for covariates $\srz$ in randomized experiments (a) yields consistent estimates despite misspecification of $\mathcal L(\pry|\prx,\prz)$ and (b) improve estimation efficiency to the extent that this adjustment captures the distribution $\mathcal L(\pry|\prx,\prz)$. This fact mirrors the conclusions of Theorem~\ref{thm:power}. The asymptotic variance of the regression-based estimator for the average treatment effect in a completely randomized experiment is a standard result, but we are unaware of a quantitative expression of the asympotic efficiency of covariate-adjusted versions of the Fisher randomization test (though some insight is provided by \cite{Zhao2021}).

\section{Most powerful one-bit $p$-values for knockoffs}
\label{sec:knockoffs}

MX knockoffs \cite{CetL16} operate differently than the CRT; they simultaneously test the conditional associations of many variables with a response. Given $m$ variables $\prx_1, \dots, \prx_m$ and a response $\pry$, it is of interest to test the CI hypotheses
\begin{equation*}
H_j: \pry \independent \prx_j \mid \prx_{-j}, \quad j = 1, \dots, m.
\end{equation*}
Note that $j$ indexes variables, rather than samples. Comparing to our setup, $\prx_j$ plays the role of $\prx$ and $\prx_{-j}$ plays the role of $\prz$. In particular, we allow $\prx_j$ to be a group of variables. Like HRT, knockoffs only requires one model fit, so it too is computationally faster than the CRT. Among these three MX procedures, knockoffs is currently the most popular. We briefly review it next, and then present an optimality result in the spirit of Theorem~\ref{prop:crt-optimality}. Its proof is given in Appendix~\ref{sec:knockoffs-proofs}.

\subsection{A brief overview of knockoffs} \label{sec:knockoffs-overview}

A set of knockoff variables $\prxk = (\prxk_1, \dots, \prxk_m)$ is constructed to satisfy conditional exchangeability:
\begin{equation}
\mathcal L(\prx_j, \prxk_j | \prx_{-j}, \prxk_{-j}) = \mathcal L(\prxk_j, \prx_j | \prx_{-j}, \prxk_{-j}), \quad j = 1, \dots, m
\label{conditional-exchangeability}
\end{equation}
and conditional independence 
\begin{equation}
\pry \independent \prxk \mid \prx.
\label{knockoff-conditional-independence}
\end{equation}
Here, $\prx_j \in \mathbb R$ denotes the $j$th element of the vector $\prx \in \mathbb R^m$ and $\prx_{-j} \in \mathbb R^{m-1}$ denotes all elements except the $j$th. Also, $\srx_{i,\bullet} \in \mathbb R^{n}$, $\srx_{\bullet, j} \in \mathbb R^m$, and $\srx_{\bullet, -j} \in \mathbb R^{n \times (m-1)}$ denote the $i$th row, $j$th column, and all columns but the $j$th of the matrix $\srx \in \mathbb R^{n \times m}$. Given such a construction, a set of knockoff variables $\srxk_{i,\bullet}$ is sampled from  $\mathcal L(\prxk|\prx = \srx_{i,\bullet})$ for each $i$. Knockoff inference is then based on a form of data-carving: variables are given an ordering $\tau(1), \dots, \tau(m)$ determined arbitrarily from $([\srx, \srxk], \sry)$ as long as $\srx_{\bullet, j}$ and $\srxk_{\bullet, j}$ are treated symmetrically. Variables are then tested in that order based on \textit{one-bit $p$-values} $p_j$ measuring the contrast between the strength of association between $\srx_{\bullet, j}$ and $\sry$ and that between $\srxk_{\bullet, j}$ and $\sry$. Given any statistic $T_j([\srx, \srxk], \sry)$ measuring the strength of association between $\srx_j$ and $\sry$, define the one-bit $p$-value
\begin{equation}
p_j([\srx, \srxk], \sry) \equiv 
\begin{cases}
\frac12, \quad &\text{if } T_j([\srx, \srxk], \sry) > T_j([\srx, \srxk]_{\text{swap}(j)}, \sry);  \\
1, \quad &\text{if } T_j([\srx, \srxk], \sry) \leq T_j([\srx, \srxk]_{\text{swap}(j)}, \sry).
\end{cases}
\label{one-bit-pvalue}
\end{equation}
Here, $[\srx, \srxk]_{\text{swap}(j)}$ is defined as the result of swapping $\srx_{\bullet, j}$ with $\srxk_{\bullet, j}$ in $[\srx, \srxk]$ while keeping all other columns in place. A set of variables with guaranteed false discovery rate control is chosen via the ordered testing procedure \textit{Selective SeqStep} \cite{BC15}, applied to the $p$-values $p_j$ in the order $\tau$.

\subsection{Most powerful one-bit $p$-value}

It is harder to analyze the power of knockoffs than that of the CRT for several reasons. Knockoffs is fundamentally a \textit{multiple} testing procedure, coupling the analysis of $H_j$ across variables $j$. Furthermore, the qualities of the ordering $\tau$ and of the one-bit $p$-values $p_j$ both contribute to the power of knockoffs. Due to these challenges, no optimality results are currently available for knockoffs. We take a first step in this direction by exhibiting the test statistics $T_j$ that lead to most powerful one-bit $p$-values against a point alternative. 

\begin{theorem} \label{prop:knockoff-optimality}
	Let $\bar{\mathcal L}$ be a fixed alternative distribution for $(\prx,\pry)$, with $\bar{\mathcal L}(\pry|\prx) = \bar f(\pry|\prx)$. Define the likelihood statistic
	\begin{equation}
	T_j^{\textnormal{opt}}([\srx, \srxk], \sry) \equiv \prod_{i = 1}^n\bar f(\sry_i|\srx_{i,\bullet}).
	\label{log-likelihood-ratio-knockoffs}
	\end{equation}
	Assuming that ties do not occur, that is
	\begin{equation}
	\mathbb P_{\bar {\mathcal L}}[T^{\textnormal{opt}}_j([\srx, \srxk], \sry) = T^{\textnormal{opt}}_j([\srx, \srxk]_{\textnormal{swap}(j)}, \sry), \srx_{\bullet, j} \neq \srxk_{\bullet, j}] = 0, 
	\label{nondegeneracy-assumption}
	\end{equation}
	we have that the above likelihood statistic yields an optimal one-bit $p$-value:
	\begin{equation}
	T_j^{\textnormal{opt}} \in \underset{T_j}{\arg \max}\ \mathbb P[T_j([\srx, \srxk], \sry) > T_j([\srx, \srxk]_{\textnormal{swap}(j)}, \sry)].
	\label{unconditional-knockoff-optimality}
	\end{equation}
\end{theorem}

This theorem is proved in Appendix~\ref{sec:knockoffs-proofs}. The reader observes that $T_j^{\text{opt}}$ is not a function of the knockoff variables or of the index $j$, which may at first seem paradoxical. Recall from the definition~\eqref{one-bit-pvalue}, however, that the one-bit $p$-value compares the test statistic on the original augmented design $[\srx, \srxk]$ and its swapped version $[\srx, \srxk]_{\textnormal{swap}(j)}$. Therefore, the optimal one-bit $p$-value checks whether the original $j$th variable $\srx_{\bullet, j}$ fits with the rest of the data better than does its knockoff $\srxk_{\bullet, j}$. Therefore, the optimal one-bit $p$-value is in fact a function of the knockoffs as well as the index $j$. A simple way of operationalizing Theorem~\ref{prop:knockoff-optimality} is to fit a model $\widehat f(\pry|\prx)$ based on $([\srx, \srxk], \sry)$ in any way that treats original variables and knockoffs symmetrically, and then defining $T_j([\srx, \srxk], \sry) \equiv \widehat f(\sry|\srx)$. The above result continues to hold when $\prx_j$ is a \textit{group} of variables, giving a clean way to combine evidence across multiple variables. A conditional version of the optimality statement~\eqref{unconditional-knockoff-optimality} holds; see equation~\eqref{knockoff-conditional-optimality} in the appendix.

Theorem~\ref{prop:knockoff-optimality} requires that ties occur with probability zero~\eqref{nondegeneracy-assumption}. Proposition~\ref{prop:nondegeneracy-knockoffs} below (proved in Appendix~\ref{sec:proofs-sec45}) states that this nondegeneracy condition holds if either $\pry|\prx$ or $\prx_j|\prx_{-j}, \prxk$ has a continuous distribution.
\begin{proposition}\label{prop:nondegeneracy-knockoffs}
	Suppose $\bar{\mathcal L}(\pry|\prx) = g_{\bm\eta}$, where $\bm \eta = \prx_j \beta_j + f_{-j}(\prx_{-j})$ and $g_\eta$ is a one-dimensional exponential family with natural parameter $\eta$ and strictly convex, continuous log partition function $\psi$. Suppose also that $\prx_j, \beta_j \in \mathbb R$, with $\beta_j \neq 0$. The nondegeneracy condition~\eqref{nondegeneracy-assumption} holds if either 
	\begin{enumerate}
		\item $\prx_{j}|\prx_{-j}, \prxk$ has a density for each $\prx_{-j}, \prxk$, or
		\item $g_\eta$ has a density,
	\end{enumerate}
	where the densities are with respect to the Lebesgue measure.
\end{proposition}

Finally, we remark that there are a few existing power analyses for knockoffs, all in high-dimensional asymptotic regimes and assuming lasso-based test statistics. Weinstein et al \cite{Weinstein2017} analyze the power of a knockoffs variant in the case of independent Gaussian covariates, while Liu and Rigollet \cite{Liu2019} and Fan et al \cite{Fan2020} study conditions for consistency under correlated designs. Our finite-sample optimality result for the likelihood statistic is complementary to these previous works. Recently, Theorem~\ref{prop:knockoff-optimality} inspired a more powerful variant of knockoffs based on \textit{masked likelihood ratio} statistics, which comes with a different kind of optimality guarantee \cite{Spector2022}.

\section{Discussion}
\label{sec:discussion}

In this paper, we gave some answers to the theoretical questions posed in the introduction. We presented the first finite-sample optimality results in the MX framework and explicitly quantified how the performance of the underlying machine learning procedure impacts the asymptotic power of the CRT. Along the way, we exhibited a weakened form of the MX assumption and a resampling-free methodology valid under only this assumption.


The MX framework is just one setting where black-box prediction methods have been recently employed for the purpose of more powerful statistical inference. Other examples include conformal prediction \cite{FoygelBarber2019}, classification-based two-sample testing \cite{Kim2020} and data-carving based multiple testing \cite{lei2016adapt}. These methods employ machine learning algorithms to create powerful test statistics, calibrating them for valid inference with no assumptions about the method used. However, the more accurate the learned model, the more powerful the inference. Our finite-sample and asymptotic power results explicitly tie the error of the learning algorithm to the power of the test, and thus put this common intuition on a quantitative foundation and may thus help inform the choice and design of machine learning methods used for inferential goals. 

Another set of connections we highlighted throughout the paper is to causal inference and semiparametric estimation. The MX CI problem has strong similarities to the problem of testing Fisher's strong null in a randomized experiment with potentially non-binary treatment and known propensity function. Furthermore, the CRT is similar in spirit to the Fisher randomization test. We believe these connections can be further leveraged to address problems in the MX framework that remain open. For example, consider the situation when the MX assumption is only approximately correct. This is analogous to the situation in observational studies, where the propensity score/function must be estimated. There is a vast literature on this topic based on ``double robustness/machine-learning''~\cite{Chernozhukov2018} or targeted learning~\cite{VanderLaan2011}. Similar ideas may help relax the MX assumption \cite{Huang2019} or study robustness to its misspecification \cite{Barber2018}. Another topic that has received little attention in the MX community is that of estimation (with the exception of \cite{Zhang2020}). Causal inference is a rich source of meaningful estimands (such as the \textit{dose response function} \cite{Hirano2004}) and estimators (such as the proposal of Kennedy et al. \cite{Kennedy2017} for doubly-robust dose response function estimation). Such ideas may be directly relevant to the MX framework.

Much still remains to be done to systematically understand the theoretical properties of MX methods. One interesting direction is to analyze the case when $\widehat g_n$ is learned on the same data as is used for testing. We saw in Section~\ref{sec:comparison-to-existing-results-4} that Theorem~\ref{thm:power} extends to lasso-based estimators $\widehat g_n$ learned in-sample, but the generality of such results remains an open question. It would also be interesting to consider alternatives beyond the linear model~\eqref{eq:linearity}. A natural next step would be to consider generalized linear models. Furthermore, the connections to causal inference referenced above are tantalizing and deserve a dedicated treatment. Finally, we hope that these new theoretical insights about MX methods will lead to improved methodologies that are both statistically and computationally efficient, along the lines of the CRT variants discussed in this paper and in recent work \cite{Liu2020}.

\section*{Acknowledgments}
We thank Asaf Weinstein, Timothy Barry, and Stephen Bates for detailed comments on earlier versions of the manuscript, as well as Ed Kennedy and Larry Wasserman for discussions of the connections to causal inference. We also thank two anonymous referees for constructive feedback that greatly helped us improve the manuscript.

\printbibliography

\appendix

\section{Proofs of Theorem~\ref{prop:crt-optimality}} \label{sec:proofs-sec2}

\begin{proof}
Let $\phi$ be any test satisfying conditional validity property~\eqref{eq:conditional-validity}. Let $\mathcal A$ be a set of pairs $(\sfy, \sfz)$, for which both $\phi$ and $\phi^{\CRT}_{T^{\textnormal{opt}}}$ have level $\alpha$ conditionally on $\sry = \sfy, \srz = \sfz$. By assumption, $\mathbb P[(\sry, \srz) \in \mathcal A] = 1$. Now, fix realizations $(\sfy, \sfz) \in \mathcal A$. We first claim that the conditional power of $\phi$ is bounded above by that of $\phi^{\CRT}_{T^{\textnormal{opt}}}$, i.e.
	\begin{equation}
		\mathbb E_{\bar{\mathcal L}}[\phi(\srx,\sfy,\sfz)|\sry = \sfy, \srz = \sfz] \leq \mathbb E_{\bar{\mathcal L}}[\phi^{\CRT}_{T^\textnormal{opt}}(\srx, \sfy, \sfz)|\sry = \sfy, \srz = \sfz]
		\label{conditional-optimality}
	\end{equation}
	In the conditional problem, the alternative $\bar{\mathcal L}$ induces the following distribution for $\srx$:
	\begin{equation}
		\begin{split}
			\bar{\mathcal L}(\srx = \sfx|\sry = \sfy, \srz = \sfz) = \prod_{i = 1}^n  f^*(\sfx_i|\sfz_i)\tfrac{\bar f(\sfy_i|\sfx_i, \sfz_i)}{\bar f(\sfy_i|\sfz_i)},
			\label{conditional-alternative}
		\end{split}
	\end{equation}
	where
	\begin{equation*}
	\bar f(\sfy_i|\sfz_i) \equiv \int \bar f(\sfy_i|\sfx_i, \sfz_i)f^*(\sfx_i|\sfz_i)d\sfx_i.
	\end{equation*}	
	The conditional problem is therefore a test of 
	\begin{equation*}
		\begin{split}
			&H_0: \mathcal L(\srx = \sfx|\sry = \sfy, \srz = \sfz) = \prod_{i = 1}^n  f^*(\sfx_i | \sfz_i) \quad \text{versus} \\
			&H_1: \mathcal L(\srx = \sfx|\sry = \sfy, \srz = \sfz) = \prod_{i = 1}^n  f^*(\sfx_i|\sfz_i)\tfrac{\bar f(\sfy_i|\sfx_i, \sfz_i)}{\bar f(\sfy_i|\sfz_i)}.
		\end{split}
	\end{equation*}
	This is a simple testing problem, with point null and point alternative. By the Neyman-Pearson lemma, the most powerful test is the one that rejects for large values of the likelihood ratio
	\begin{equation}
 \frac{\prod_{i = 1}^n f^*(\sfx_i|\sfz_i)\frac{\bar f(\sfy_i|\sfx_i, \sfz_i)}{\bar f(\sfy_i|\sfz_i)}}{\prod_{i = 1}^n f^*(\sfx_i|\sfz_i)} = \prod_{i = 1}^n \frac{\bar f(\sfy_i|\sfx_i, \sfz_i)}{\bar f(\sfy_i|\sfz_i)} \propto T^{\text{opt}}(\sfx, \sfy, \sfz),
	\label{eq:likelihood-ratio-derivation}
	\end{equation}
	verifying the conditional optimality claim~\eqref{conditional-optimality}. To obtain the unconditional claim~\eqref{unconditional}, we take an expectation over $\sry, \srz$ and use the fact that $\mathbb P[(\sry, \srz) \in \mathcal A] = 1$:
	\begin{equation}
		\begin{split}
			\mathbb E_{\bar{\mathcal L}}[\phi(\srx,\sry,\srz)] &= \mathbb E_{\bar{\mathcal L}}[\phi(\srx,\sry,\srz) \mid (\sry, \srz) \in \mathcal A] \\
			&= \mathbb E_{\bar{\mathcal L}}[\mathbb E_{\bar{\mathcal L}}[\phi(\srx,\sry,\srz) \mid \sry,\srz] \mid (\sry, \srz) \in \mathcal A] \\
			&\leq \mathbb E_{\bar{\mathcal L}}\left[\mathbb E_{\bar{\mathcal L}}[\phi^{\CRT}_{T^\textnormal{opt}}(\srx,\sry,\srz)\mid\sry,\srz] \mid (\sry, \srz) \in \mathcal A\right] \\
			&= \mathbb E_{\bar{\mathcal L}}\left[\phi^{\CRT}_{T^\textnormal{opt}}(\srx,\sry,\srz) \mid (\sry, \srz) \in \mathcal A\right] = \mathbb E_{\bar{\mathcal L}}[\phi^{\CRT}_{T^\textnormal{opt}}(\srx, \sry, \srz)].
		\end{split}	
	\end{equation}
This completes the proof.
\end{proof}

\section{Simulation: Finite sample error control of the MX(2) $F$-test}  \label{sec:simulations}

In this section, we examine via numerical simulation the Type-I error control of the MX(2) $F$-test in finite samples, both if $\widehat g_n$ is fit out of sample (the case covered by Theorem~\ref{thm:asymptotic-alpha-level}) and if $\widehat g_n$ is fit in sample (conjectured). Code to reproduce the simulation is available online at \url{https://github.com/ekatsevi/crtpower-manuscript}.

\paragraph{Simulation setup.}

Recall that the MX(2) $F$-test is equivalent to the dCRT in finite samples when $\mathcal L_n(\prx|\prz)$ is Gaussian. Therefore, to test the Type-I error control of the MX(2) $F$-test in a nontrivial setting, we instead consider a discrete distribution for $(\prx,\prz)$. In particular, we sample $(\prx, \prz)$ from a Markov chain, as described next. (Such a Markovian setup has often been employed in MX analyses of GWAS studies \cite{SetC17, SetS19, Bates2020}, motivated by recombination models from population genetics.)

Let's assume for simplicity that $\text{dim}(\prx) = 1$. Define $(\prx,\prz) \in \{0,1\}^{1+p}$ to have the distribution of a Markov chain with 
\begin{equation*}
	\text{initial state } \prx \sim \text{Ber}(\pi_{\text{init}}) \text{ and transition matrix } \begin{pmatrix}1-\pi_\text{flip} & \pi_{\text{flip}} \\  \pi_{\text{flip}} &  1-\pi_{\text{flip}}\end{pmatrix}.
\end{equation*}
More explicitly, we have
\begin{equation*}
	\mathbb P[\prx = x, \prz = z] = \pi_{\text{init}}^{x}(1-\pi_{\text{init}})^{1-x} \pi_{\text{flip}}^{\mathbbm 1(z_1 \neq x)} (1-\pi_{\text{flip}})^{\mathbbm 1(z_1 = x)} \prod_{j = 2}^p \pi_{\text{flip}}^{\mathbbm 1(z_j \neq z_{j-1})} (1-\pi_{\text{flip}})^{\mathbbm 1(z_j = z_{j-1})}.
\end{equation*}
The parameters $(\pi_{\text{init}}, \pi_{\text{flip}})$ describe the distribution of $\prx|\prz$ and are assumed known. Furthermore, let the response $\pry$ be distributed as a random effects model in $\prz$:
\begin{equation*}
	\pry = \prz^T \bm \gamma + \peps, \quad \bm \gamma \sim N(0, \sigma^2_{\gamma}I_p),\ \peps \sim N(0, \sigma^2_\eps I_n).
\end{equation*}
Thus, all simulations are conducted under the null hypothesis $H_0: \prx \independent \pry \mid \prz$. The signal-to-noise ratio in this relationship is defined via
\begin{equation*}
	\text{SNR} = \frac{\mathbb E[\|\prz\|^2]\sigma^2_\gamma}{\sigma^2_\eps}.
\end{equation*}
Suppose we have $n_{\text{train}}$ and $n_{\text{test}}$ training and test samples, respectively. Then, the function $\widehat g_n$ is defined by running a 10-fold cross-validated ridge regression of $\sry$ on $\srz$ using either the $n_{\text{train}}$ training samples (out of sample training) or the $n_{\text{test}}$ test samples (in sample training) and then the statistic $U_n(\srx, \sry, \srz)$ is computed using the $n_{\text{test}}$ test samples.

\paragraph{Simulation parameters.}

All simulations were run with 
\begin{equation}
	n_{\text{train}} = 100; \quad \pi_{\text{init}} = 0.1; \quad \pi_{\text{flip}} = 0.1; \quad \sigma^2_\eps = 1. 
	\label{fixed-parameters}
\end{equation}
On the other hand, the three parameters $(n_{\text{test}}, \text{SNR}, p)$ were varied as follows:
\begin{equation*}
	n_{\text{test}} \in \{10, 25, \textbf{100}\}; \quad  \text{SNR} \in \{0, \textbf{1}, 5\}; \quad p \in \{20, 100, \textbf{500}\}.
\end{equation*}
The bolded values above represent the \textit{default values} for each parameter. Each of the three parameters was varied while keeping the other two parameters at their default values, giving a total of nine simulation settings. For each simulation setting, the training data were generated just once, since Theorem~\ref{thm:asymptotic-alpha-level} implicitly conditions on the training data. The entire test data $(\srx, \sry, \srz)$  were sampled 1000 times to generate the null distribution of $U_n(\srx, \sry, \srz)$. 

\paragraph{Simulation results.}

For each of the nine simulation settings, we produce normal QQ plots of the  $z$-statistics $U_n(\srx, \sry, \srz)$ based on out of sample or in sample training (Figures~\ref{fig:out-of-sample} and~\ref{fig:in-sample}, respectively). When $\widehat g_n$ is fit out of sample (Figure~\ref{fig:out-of-sample}), we see good calibration in most cases. In particular, the test sample size impacts calibration, but the SNR and the dimension do not. The test statistic's null distribution shows some inflation for the small sample size of $n_{\text{test}} = 10$, but is already well-calibrated starting with $n_{\text{test}} = 25$. Therefore, the asymptotic Type-I error control proved in Theorem~\ref{thm:asymptotic-alpha-level} extends to modest sample sizes as well. Furthermore, when $\widehat g_n$ is fit in sample (Figure~\ref{fig:in-sample}), we observe calibration that is as good as when $\widehat g_n$ is fit out of sample. This suggests that we may apply the MX(2) $F$-test even with in-sample-estimated $\widehat g_n$. We must bear in mind, however, that different choices of the fixed parameters~\eqref{fixed-parameters} may alter these conclusions. In particular, smaller $\pi_{\text{init}}$ leads to more discreteness in $\srx$ and therefore slower convergence to normality.

\clearpage

\begin{figure}[h!]
	\includegraphics[width = \textwidth]{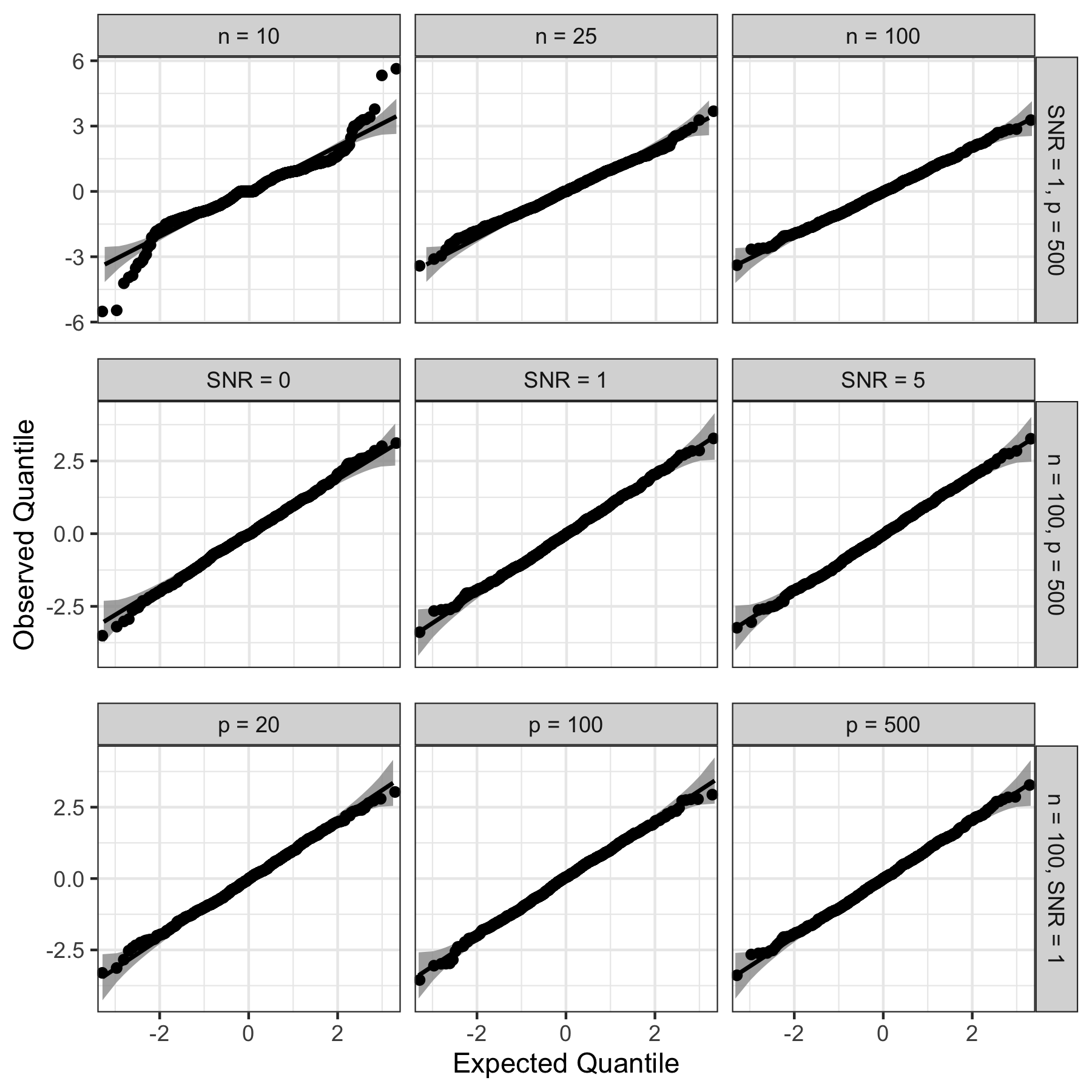}
	\caption{Distributions of 1000 samples of $U_n(\srx, \sry, \srz)$ each from nine simulation settings under the null, where $\widehat g_n$ is learned \textit{out of sample}.}
	\label{fig:out-of-sample}
\end{figure}

\clearpage

\begin{figure}[h!]
	\includegraphics[width = \textwidth]{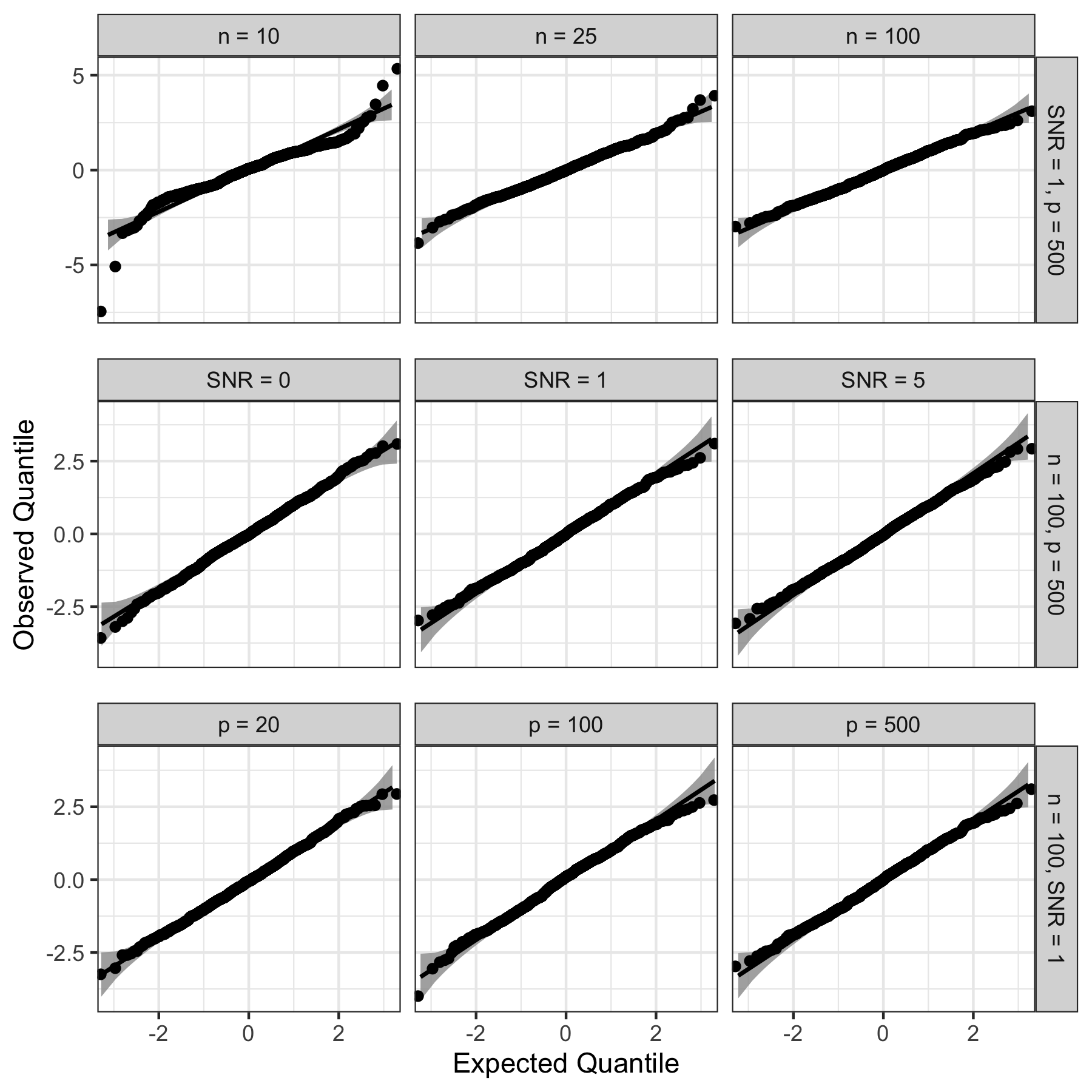}
	\caption{Distributions of 1000 samples of $U_n(\srx, \sry, \srz)$ each from nine simulation settings under the null, where $\widehat g_n$ is learned \textit{in sample}.}
	\label{fig:in-sample}
\end{figure}

\clearpage

\section{Proofs for Sections~\ref{sec:weakening} and~\ref{sec:asymptotic-power}} \label{sec:proofs-sec45}

\subsection{Proofs of main results}

\begin{proof}[Proof of Theorem~\ref{thm:equivalence}]
	
	First, conclusion~\eqref{convergence-2} of Lemma~\ref{lem:clt}---which applies because of the assumption $\mathcal L_n \in \mathscr L^{\textnormal{MX(2)}}(\mu_n(\cdot), \Sigma_n(\cdot)) \cap \mathscr L_n(c_1, c_2)$---states that for 
	\[
	\srxk^1_i, \srxk^2_i|\sry,\srz \overset{\text{ind}}\sim {\mathcal L_n(\prx|\prz = \srz_i)}, 
	\]
	we have the convergence
	\begin{equation}
		{U_n(\srxk^1, \sry, \srz) \choose U_n(\srxk^2, \sry, \srz)} \overset{\mathcal L_n}\rightarrow_d N\left({0 \choose 0},
		\begin{pmatrix}
			I_d & 0 \\
			0 & I_d
		\end{pmatrix}
		\right).
	\end{equation}
	By the continuous mapping theorem, we find that 
	\begin{equation}
		(T_n(\srxk^1, \sry, \srz), T_n(\srxk^2, \sry, \srz)) \overset{\mathcal L_n}\rightarrow_d  \chi^2_d \times \chi^2_d. 
	\end{equation}
	Since $\chi^2_d$ has a continuous and strictly increasing distribution function, we conclude using Lemma~\ref{lem:lucas} that $C_n(\sry,\srz)  \overset{\mathcal L_n}\rightarrow_p Q_{1-\alpha}[\chi^2_d] =  c_{d,1-\alpha}$, proving the statement~\eqref{eq:threshold-convergence}.
	
	Next, note that for any $\delta > 0$,
	\small
	\begin{equation*}
		\begin{split}
			&\mathbb P_{\mathcal L_n}[\phi^{N(\mu_n, \Sigma_n)}_n(\srx, \sry, \srz) \neq \phi^{\mathcal L_n}_n(\srx, \sry, \srz)] \\
			&\quad = \mathbb P_{\mathcal L_n}[\min(c_{d,1-\alpha},C_n(\sry, \srz))  < T_n(\srx, \sry, \srz) \leq \max(c_{d,1-\alpha},C_n(\sry, \srz))] \\
			&\quad=\mathbb P_{\mathcal L_n}[\min(c_{d,1-\alpha},C_n(\sry, \srz))  < T_n(\srx, \sry, \srz) \leq \max(c_{d,1-\alpha},C_n(\sry, \srz)), |C_n(\sry, \srz)-c_{d,1-\alpha}| \leq \delta] \\
			&\quad \quad +  \mathbb P_{\mathcal L_n}[\min(c_{d,1-\alpha},C_n(\sry, \srz))  < T_n(\srx, \sry, \srz) \leq \max(c_{d,1-\alpha},C_n(\sry, \srz)), |C_n(\sry, \srz)-c_{d,1-\alpha}| > \delta]\\
			&\quad\leq \mathbb P_{\mathcal L_n}[|T_n(\srx, \sry, \srz)-c_{d,1-\alpha}| \leq \delta] + \mathbb P_{\mathcal L_n}[|C_n(\sry, \srz)-c_{d,1-\alpha}| > \delta].
		\end{split}
	\end{equation*}
	\normalsize
	To justify the last step, suppose without loss of generality that $c_{d,1-\alpha} \leq C_n(\sry, \srz)$. Then, note that if $c_{d,1-\alpha} < T_n(\srx, \sry, \srz) \leq C_n(\sry, \srz)$ and $C_n(\sry, \srz)-c_{d,1-\alpha} \leq \delta$ then
	\begin{equation*}
		|T_n(\srx, \sry,\srz)-c_{d,1-\alpha}| = T_n(\srx, \sry,\srz)-c_{d,1-\alpha} \leq C_n(\sry, \srz)- c_{d,1-\alpha} \leq \delta.
	\end{equation*}
	Taking a $\limsup$ on both sides in the display before the last and using the convergence $C_n(\sry,\srz)  \overset{\mathcal L_n}\rightarrow_p c_{d,1-\alpha}$, we find that
	\begin{equation*}
		\begin{split}
			&\limsup_{n \rightarrow \infty}\ \mathbb P_{\mathcal L_n}[\phi^{N(\mu_n, \Sigma_n)}_n(\srx, \sry, \srz) \neq \phi^{\mathcal L_n}_n(\srx, \sry, \srz)] \\
			&\quad \leq \limsup_{n \rightarrow \infty}\ \mathbb P_{\mathcal L_n}[|T_n(\srx, \sry, \srz)-c_{d,1-\alpha}| \leq \delta]+ \limsup_{n \rightarrow \infty}\ \mathbb P_{\mathcal L_n}[|C_n(\sry, \srz)-c_{d,1-\alpha}| > \delta] \\
			&\quad = \limsup_{n \rightarrow \infty}\ \mathbb P_{\mathcal L_n}[|T_n(\srx, \sry, \srz)-c_{d,1-\alpha}| \leq \delta].
		\end{split}
	\end{equation*}
	Letting $\delta \rightarrow 0$ and using the assumption~\eqref{eq:non-accumulation}, we arrive at the claimed asymptotic equivalence~\eqref{eq:asymptotic-equivalence}. This completes the proof.
\end{proof}

\begin{proof}[Proof of Theorem~\ref{thm:asymptotic-alpha-level}]

Fix any sequence $\mathcal L_n \in \mathscr L_0^{\textnormal{MX(2)}}\cap \mathscr L_n(c_1, c_2)$. Because $\mathcal L_n \in \mathscr L_0$, we have $(\srx, \sry, \srz) \overset d = (\srxk, \sry, \srz)$, where $\srxk_i|\sry,\srz \overset{\text{ind}}\sim \mathcal L_n(\prx|\prz = \srz_i)$. By conclusion~\eqref{convergence-2} of Lemma~\ref{lem:clt}, which applies because $\mathcal L_n \in \mathscr L^{\textnormal{MX(2)}}(\mu_n(\cdot), \Sigma_n(\cdot)) \cap \mathscr L_n(c_1, c_2)$ by assumption, we have $U_n(\srx, \sry,\srz) \overset d = U_n(\srxk, \sry,\srz) \overset{\mathcal L_n}\rightarrow_d N(0, I_d)$. This verifies the asymptotic normality statement~\eqref{eq:asymptotic-normality}.

To show the asymptotic Type-I error control statement~\eqref{asymptotic-alpha-level}, it suffices to show that for any sequence $\mathcal L_n \in \mathscr L_0^{\textnormal{MX(2)}}\cap \mathscr L_n(c_1, c_2)$, we have
	\begin{equation}
		\limsup_{n \rightarrow \infty}\ \mathbb E_{\mathcal L_n}[\phi^{N(\mu_n, \Sigma_n)}_n(\srx, \sry, \srz)] \leq \alpha.
	\label{eq:pointwise-alpha-level}
	\end{equation}
By the continuous mapping theorem it follows from asymptotic normality~\eqref{eq:asymptotic-normality} that $T_n(\srx,\sry,\srz) = \|U_n(\srx, \sry,\srz)\|^2 \overset{\mathcal L_n}\rightarrow_d \chi^2_d$. Therefore,
	\begin{equation*}
		\lim_{n \rightarrow \infty}\ \mathbb E_{\mathcal L_n}[\phi^{N(\mu_n, \Sigma_n)}_n(\srx, \sry, \srz)] = \lim_{n \rightarrow \infty}\ \mathbb P_{\mathcal L_n}[T_n(\srx,\sry,\srz) > c_{d,1-\alpha}] = \mathbb P[\chi^2_d > c_{d,1-\alpha}] = \alpha,
	\end{equation*}
	from which the conclusion~\eqref{eq:pointwise-alpha-level} follows. This completes the proof.
\end{proof}

\begin{proof}[Proof of Theorem~\ref{thm:power}]
	
In Lemma~\ref{lem:consistency}, we show that the estimator $\widehat \rho_n$ is consistent, i.e. 
\begin{equation}
\widehat \rho_n \overset{\mathcal L_n}\rightarrow_p \overline\Sigma \beta. 
\label{estimation-consistency}
\end{equation}
Next, we derive that
\begin{equation*}
T_n(\srx, \sry, \srz) = \|\sqrt n\widehat S_n^{-1}\widehat \rho_n\|^2 = \|\sqrt n\widehat S_n^{-1}S_n S_n^{-1}\widehat \rho_n\|^2 \geq \left(\sqrt{n}\lambda_{\min}(\widehat S_n^{-1}S_n)\lambda_{\min}(S_n^{-1}) \|\widehat \rho_n\|\right)^2.
\end{equation*}
Now, we have $\widehat S_n^{-1}S_n \overset{\mathcal L_n}\rightarrow _p  I_d$ by conclusion~\eqref{eq:conv-s-neg-1} of Lemma~\ref{lem:aux}, so the continuous mapping theorem implies that $\lambda_{\min}(\widehat S_n^{-1}S_n) \overset{\mathcal L_n}\rightarrow _p 1$. Furthermore, $\inf_n \lambda_{\min}(S_n^{-1}) > 0$ by conclusion~\eqref{eq:s-n-2-limit-a} of Lemma~\ref{lem:fourth-moment}. Finally, $\|\widehat \rho_n\| \overset{\mathcal L_n}\rightarrow _p \|\overline \Sigma \beta\|$ by equation~\eqref{estimation-consistency}, and 
\[
\|\overline \Sigma \beta\| \geq \lambda_{\min}(\overline \Sigma)\|\beta\| = \|\overline \Sigma^{-1}\|^{-1}\|\beta\| \geq \left(\sup_n \|\overline \Sigma_n^{-1}\|\right)^{-1}\|\beta\| > 0,
\]
since $\beta \neq 0$ by assumption and assumptions~\eqref{eq:s-n-inverse-assump} and \eqref{eq:limits} imply that $\|\overline \Sigma^{-1}\| \leq \sup_n \|\overline \Sigma_n^{-1}\| < \infty$. Putting these facts together implies that $\sqrt{n}\lambda_{\min}(\widehat S_n^{-1}S_n)\lambda_{\min}(S_n^{-1}) \|\widehat \rho_n\|\overset{\mathcal L_n}\rightarrow _p \infty$, and therefore $T_n(\srx, \sry, \srz) \overset{\mathcal L_n}\rightarrow _p \infty$. Hence,
\begin{equation}
\begin{split}
\mathbb E_{\mathcal L_n}[\phi_n^{N(\mu_n, \Sigma_n)}(\srx, \sry, \srz)] = \mathbb P_{\mathcal L_n}[T_n(\srx, \sry, \srz) > c_{d,1-\alpha}] \rightarrow 1.
\end{split}
\end{equation}
The fact that $T_n(\srx, \sry, \srz) \overset{\mathcal L_n}\rightarrow _p \infty$ also implies that $\limsup_{n \rightarrow \infty}\ \mathbb P_{\mathcal L_n}[|T_n(\srx, \sry, \srz)-c_{d,1-\alpha}| \leq \delta] = 0$ for any $\delta > 0$. Hence, the condition~\eqref{eq:non-accumulation} of Theorem~\ref{thm:equivalence} is satisfied, so the conclusion~\eqref{eq:asymptotic-equivalence} implies that
\begin{equation*}
\lim_{n \rightarrow \infty}\mathbb E_{\mathcal L_n}[\phi_n^{\mathcal L_n}(\srx, \sry, \srz)] = \lim_{n \rightarrow \infty}\mathbb E_{\mathcal L_n}[\phi_n^{N(\mu_n, \Sigma_n)}(\srx, \sry, \srz)] = 1.
\end{equation*} 
Thus, we have shown the claimed consistency~\eqref{eq:consistency}, so we have finished the proof of part (a) of the theorem.
	
To prove part (b), we claim that it suffices to establish that
\begin{equation}
	T_n(\srx, \sry, \srz)\overset{\mathcal L_n}\rightarrow_d \chi^2_d(\|(\sigma^2I_d +\mathcal E^2)^{-1/2}\overline \Sigma^{1/2} h\|^2).
	\label{eq:convergence-of-T}
\end{equation}
Indeed, the limiting power of the MX(2) $F$-test would directly follow from this statement. To establish that the CRT has the same limiting power, by Theorem~\ref{thm:equivalence} it suffices to verify the non-accumulation condition~\eqref{eq:non-accumulation}. Letting $T$ be the limiting distribution in claim~\eqref{eq:convergence-of-T}, this claim implies that for any $\delta > 0$,
\begin{equation*}
 \lim_{n \rightarrow \infty}\ \mathbb P_{\mathcal L_n}[|T_n(\srx, \sry, \srz)-c_{d,1-\alpha}| \leq \delta] = \mathbb P[|T - c_{d,1-\alpha}| \leq \delta].
\end{equation*}
Because $T$ has a continuous distribution function, the limit above tends to zero. Therefore, it is indeed sufficient to verify the claimed convergence~\eqref{eq:convergence-of-T}. This statement, in turn, will follow if we prove that
	\begin{equation}
		U_n(\srx, \sry, \srz) \overset{\mathcal L_n}\rightarrow_d N((\overline \Sigma^{1/2}(\sigma^2I_d +\mathcal E^2)\overline \Sigma^{1/2})^{-1/2}\overline \Sigma h, I_d).
		\label{eq:U-conv}
	\end{equation}
	Indeed, note that
	\begin{equation*}
		h^T \overline \Sigma (\overline \Sigma^{1/2}(\sigma^2I_d +\mathcal E^2)\overline \Sigma^{1/2})^{-1}\overline \Sigma h = h^T \overline \Sigma^{1/2} (\sigma^2I_d +\mathcal E^2)^{-1}\overline \Sigma^{1/2} h = \|(\sigma^2I_d +\mathcal E^2)^{-1/2}\overline \Sigma^{1/2} h\|^2.
	\end{equation*}
	
	To show the statement~\eqref{eq:U-conv}, we first rewrite $U_n(\srx, \sry,\srz)$ as follows:
	\begin{equation*}
		\begin{split}
			&U_n(X,Y,Z) \\
			&\quad= \frac{\widehat S^{-1}_n}{\sqrt{n}}\sum_{i = 1}^n ((X_i - \mu_n(Z_i))^T \beta_n + \eps_i + \bar g_n(\srz_i) - \widehat g_n(\srz_i))(\srx_i - \mu_n(\srz_i)) \\
			&\quad= \frac{\widehat S^{-1}_n}{n}\sum_{i = 1}^n (\srx_i - \mu_n(\srz_i))(X_i - \mu_n(Z_i))^T h_n + \frac{\widehat S^{-1}_n}{\sqrt{n}}\sum_{i = 1}^n (Y'_i - \widehat g_n(Z_i))(\srx_i - \mu_n(\srz_i)) \\
			&\quad \equiv A_n + B_n,
		\end{split}
	\end{equation*}
	where $Y'_i \equiv \bar g_n(\srz_i) + \seps_i.$ It therefore suffices to show that
	\begin{equation}
		A_n \overset{\mathcal L_n}\rightarrow_p (\overline \Sigma^{1/2}(\sigma^2 I_d + \mathcal E^2)\overline \Sigma^{1/2})^{-1/2}\overline \Sigma h\quad \text{and} \quad B_n \overset{\mathcal L_n}\rightarrow_d N(0, I_d).
		\label{eq:sufficient}
	\end{equation}
	
	By conclusion~\eqref{eighth-moment} of Lemma~\ref{lem:fourth-moment}, there exist $c_1, c_2$ for which $\mathcal L_n \in \mathscr L(c_1, c_2)$ for each $n$. Therefore, we can apply Lemma~\ref{lem:aux} to conclude that 
	\begin{equation}
		\widehat S_n^{-1}S_n  \overset{\mathcal L_n}\rightarrow_p I_d. 
		\label{observation-1}
	\end{equation}
	By conclusion~\eqref{eq:s-n-2-limit} of Lemma~\ref{lem:fourth-moment}, we have that $S_n^2 \rightarrow \overline \Sigma^{1/2}(\sigma^2 I_d + \mathcal E^2)\overline \Sigma^{1/2}$, so 
	\begin{equation}
		S_n^{-1} \rightarrow (\overline \Sigma^{1/2}(\sigma^2 I_d + \mathcal E^2)\overline \Sigma^{1/2})^{-1/2}.
		\label{observation-2}
	\end{equation} 
	Now, we apply the WLLN to find the limit of $A_n$. Since $(\srx_i - \mu_n(\srz_i))(X_i - \mu_n(Z_i))^T$ has expectation $\overline \Sigma$ and second moment uniformly bounded by the eighth moment assumption~\eqref{eq:eighth-moment-assump-1}, we can apply the weak law of large numbers as well as the statements~\eqref{observation-1} and \eqref{observation-2} to conclude that
	\begin{equation*}
		A_n = (\widehat S^{-1}_nS_n)\frac{S_n^{-1}}{n}\sum_{i = 1}^n (\srx_i - \mu_n(\srz_i))(X_i - \mu_n(Z_i))^T h_n \overset{\mathcal L_n}\rightarrow_p (\overline \Sigma^{1/2}(\sigma^2 I_d + \mathcal E^2)\overline \Sigma^{1/2})^{-1/2}\overline \Sigma h.
	\end{equation*}
	Next, we seek to find the limit of $B_n$. Defining $\pry', S'^2_n, \mathcal L'_n$ according to \eqref{prime-definitions} below, we may rewrite
	\begin{equation}
		B_n = (\widehat S_n^{-1}S_n)(S_n^{-1}S'_n)\frac{S'^{-1}_n}{\sqrt{n}}\sum_{i = 1}^n (Y'_i - \widehat g_n(Z_i))(\srx_i - \mu_n(\srz_i)).
	\end{equation}
	By conclusion~\eqref{eq:s-n-2-limit} of Lemma~\ref{lem:fourth-moment}, we have
	\begin{equation}
		\quad S_n^{-1}S'_n \rightarrow I_d.
		\label{display-1}
	\end{equation}
	Furthermore, conclusion~\eqref{eighth-moment} of Lemma~\ref{lem:fourth-moment} gives $\mathcal L'_n \in \mathscr L^{\textnormal{MX(2)}}(\mu_n(\cdot), \Sigma_n(\cdot)) \cap \mathscr L_n(c_1, c_2)$. Therefore, $\mathcal L'_n$ satisfies the assumptions of Lemma~\ref{lem:clt}, statement~\eqref{convergence-1} of which gives
	\begin{equation}
		\begin{split}
			\frac{S'^{-1}_n}{\sqrt{n}}\sum_{i = 1}^n (Y'_i - \widehat g_n(Z_i))(\srxk_i - \mu_n(\srz_i)) \overset{\mathcal L_n}\rightarrow_d N(0, I_d).
			\label{display-2}
		\end{split}
	\end{equation}
	Furthermore, $\mathcal L'_n \in \mathscr L_0$ implies that $(\prx, \pry', \prz) \overset d = (\prxk, \pry', \prz)$, which together with the convergence~\eqref{display-2} implies that
	\begin{equation}
		\begin{split}
			\frac{S'^{-1}_n}{\sqrt{n}}\sum_{i = 1}^n (Y'_i - \widehat g_n(Z_i))(\srx_i - \mu_n(\srz_i)) \overset{\mathcal L_n}\rightarrow_d N(0, I_d).
			\label{display-3}
		\end{split}
	\end{equation}
	Finally, putting together displays~\eqref{observation-1}, \eqref{display-1} and~\eqref{display-3} yields that $B_n \overset{\mathcal L_n}\rightarrow_d N(0, I_d)$. This verifies the claimed convergences~\eqref{eq:sufficient} and therefore completes the proof.
\end{proof}

\begin{proof}[Proof of Corollary~\ref{cor:lasso}]

First we verify the statement~\eqref{eq:bm-result}. To this end, first note that
\begin{equation}
	\begin{split}
		\mathcal E_n^2 &=  \mathbb E_{\mathcal L_n}[(\widehat g_n(\prz) - \bar g_n(\prz))^2\overline \Sigma_n^{-1/2}\Sigma_n(\prz)\overline \Sigma_n^{-1/2} \mid \srx_{\text{train}}, \sry_{\text{train}}, \srz_{\text{train}}] \\
		&= \mathbb E_{\mathcal L_n}[(\widehat g_n(\prz) - g_n(\prz))^2 \mid \srx_{\text{train}}, \sry_{\text{train}}, \srz_{\text{train}}] \\
		&= \mathbb E_{\mathcal L_n}[(\widehat \gamma_{\pi n} - \gamma_n)^T \prz \prz^T (\widehat \gamma_{\pi n} - \gamma_n) \mid \srx_{\text{train}}, \sry_{\text{train}}, \srz_{\text{train}}] \\
		&= (\widehat \gamma_{\pi n} - \gamma_n)^T\mathbb E_{\mathcal L_n}[ \prz \prz^T](\widehat \gamma_{\pi n} - \gamma_n) \\
		&= \|\widehat \gamma_{\pi n} - \gamma_n\|^2.
		\label{eq:error}
	\end{split}
\end{equation}
The second equality holds because for $(\prx, \prz)$ jointly Gaussian, $\Sigma_n(\prz)$ is constant in $\prz$, so $\overline \Sigma_n^{-1/2}\Sigma_n(\prz)\overline \Sigma_n^{-1/2} = 1$.
Therefore, the variance-weighted mean-squared error $\mathcal E_n^2$ of $\widehat g_n$ reduces to the squared error in the estimate $\widehat \gamma_{\pi n}$. To obtain the limit of the latter quantity, we appeal to Bayati and Montanari's Corollary 1.6 \cite{Bayati2011}. To verify the conditions of this corollary, it suffices to verify part (b) of their Definition 1: that the empirical distribution of the noise terms $\seps'_i \equiv \sry_i - \srz_i^T \gamma_n = \srx_i \beta_n + \seps_i$ in the training set (say $1 \leq i \leq \pi n$) converges weakly to a random variable $\Lambda$ and $\frac{1}{\pi n}\sum_{i = 1}^{\pi n} \eps'^2_i \rightarrow \mathbb E[\Lambda^2]$. These statements hold almost surely in the training data by the strong law of large numbers if we assume without loss of generality that $\srx_i$ and $\seps_i$ are both defined as the first $\pi n$ elements of infinite i.i.d. sequences with distributions $N(0,1)$ and $N(0,\sigma^2)$, respectively. Therefore, Bayati and Montanari's Corollary 1.6 gives 
\begin{equation}
	\frac{1}{p}\|\sqrt{\pi n}\widehat \gamma_{\pi n} - \sqrt{\pi n}\gamma_n\|^2 \overset{\text{a.s.}}\rightarrow \pi\delta(\tau_*^2 - \sigma^2),
\end{equation}
where the almost sure statement is with respect to the training data $(\srx_{\text{train}}, \sry_{\text{train}}, \srz_{\text{train}})$. Since $\pi n/p \rightarrow \pi \delta$, we can cancel these terms from the above equation to obtain
\begin{equation}
	\|\widehat \gamma_{\pi n} - \gamma_n\|^2 \overset{\text{a.s.}}\rightarrow \tau_*^2 - \sigma^2.
	\label{eq:BM11}
\end{equation}
Putting together equations~\eqref{eq:error} and~\eqref{eq:BM11} gives the claimed statement~\eqref{eq:bm-result}. 

To apply this result, we must verify the assumptions of Theorem~\ref{thm:power}. The bounded inverse assumption~\eqref{eq:s-n-inverse-assump} holds because $\overline \Sigma_n = 1$ for all $n$ in the orthogonal design setting. The eighth moment assumption~\eqref{eq:eighth-moment-assump-1} holds due to the boundedness of the eighth moments of Gaussian random variables. To verify the moment assumption~\eqref{eq:eighth-moment-assump-2}, we note that, almost surely in the training data,
\begin{equation*}
\begin{split}
&\sup_n\ \mathbb E_{\mathcal L_n}[(\widehat g_n(\prz)-\bar g_n(\prz))^4\|\prx - \mu_n(\prz)\|^4|\srx_{\text{train}}, \sry_{\text{train}}, \srz_{\text{train}}] \\
&\quad= \sup_n\ \mathbb E_{\mathcal L_n}[(\prz \widehat \gamma_{\pi n} - \prz \gamma_n)^4\|\prx\|^4|\srx_{\text{train}}, \sry_{\text{train}}, \srz_{\text{train}}] \\
&\quad\leq \sup_n\|\widehat \gamma_{\pi n} - \gamma_n\|^4  \mathbb E_{\mathcal L_n}[\|\prz\|^4\|\prx\|^4] \\
&\quad\leq \sup_n\|\widehat \gamma_{\pi n} - \gamma_n\|^4  \mathbb E_{\mathcal L_n}[\|\prz\|^8]^{1/2} \mathbb E_{\mathcal L_n}[\|\prx\|^8]^{1/2} \\
&\quad < \infty.
\end{split}
\end{equation*}
The last inequality holds because $\|\widehat \gamma_{\pi n} - \gamma_n\|^4$ has a finite limit according to~\eqref{eq:BM11} and because $\prz$ and $\prx$ have bounded eighth moments since they are Gaussian. Finally, we verify assumption~\eqref{eq:limits} by noting that $\overline \Sigma_n\rightarrow \overline \Sigma \equiv 1$ and $\mathcal E_n^2 \rightarrow \tau_*^2 - \sigma^2 \equiv \mathcal E$, the latter by statement~\eqref{eq:bm-result}. Therefore, Theorem~\ref{thm:power} gives 
\begin{equation*}
	\begin{split}
		\mathbb E_{\mathcal L_n}[\phi^{\mathcal L_n}_{(1-\pi)n}(\srx_{\text{test}}, \sry_{\text{test}}, \srz_{\text{test}})|\srx_{\text{train}}, \sry_{\text{train}}, \srz_{\text{train}}] \overset{\text{a.s.}}\rightarrow \mathbb P[\chi^2_1(\|\tau_*^{-1} h\sqrt{1-\pi}\|^2) > c_{1,1-\alpha}].
	\end{split}
\end{equation*}
The extra factor of $\sqrt{1-\pi}$ reflects the fact that a sample size of $(1-\pi)n$ is used for testing, so $\beta_n = h_n/\sqrt{n} = h_n\sqrt{1-\pi}/\sqrt{(1-\pi)n}$. In other words, reducing the number of samples for testing from $n$ to $(1-\pi)n$ has the effect of reducing the alternative signal strength from $h_n$ to $h_n \sqrt{1-\pi}$. Noting that $c_{1,1-\alpha} = z^2_{1-\alpha/2}$, we conclude using the dominated convergence theorem that
\begin{equation*}
	\begin{split}
		\mathbb E_{\mathcal L_n}[\varphi^{\mathcal L_n}_{n}(\srx, \sry, \srz)] &=  
		\mathbb E_{\mathcal L_n}\left[\mathbb E_{\mathcal L_n}[\phi^{\mathcal L_n}_{(1-\pi)n}(\srx_{\text{test}}, \sry_{\text{test}}, \srz_{\text{test}})|\srx_{\text{train}}, \sry_{\text{train}}, \srz_{\text{train}}]\right] \\
		&\rightarrow \mathbb P[|N(\tau_*^{-1} h\sqrt{1-\pi},1)| > z_{1-\alpha/2}].
	\end{split}
\end{equation*}
This completes the proof of the corollary.
\end{proof}

\subsection{Technical lemmas}

First, we state a lemma that gives a sufficient condition for the convergence of the CRT threshold, which follows directly from  Lemmas 2 and 3 of \cite{Wang2020b}.
\begin{lemma}[\cite{Wang2020b}] \label{lem:lucas}
	Let $\mathcal L_n$ be a sequence of laws over $(\prx,\pry,\prz)$, from which $(\srx,\sry,\srz)$ are sampled. Furthermore, for each $i$ sample two independent copies
	\begin{equation}
		\srxk_i^1,\srxk_i^2 \overset{\text{i.i.d.}}\sim \mathcal L_n(\prx|\prz = \srz_i) \quad \text{such that, given } \srz,  (\srxk_1^1,\srxk_1^2) \independent \cdots \independent (\srxk_n^1,\srxk_n^2) \independent \sry.
		\label{eq:two-resamples}
	\end{equation}
	Suppose that $T_n(\srx,\sry,\srz)$ is a test statistic satisfying 
	\begin{equation}
		(T_n(\srxk^1,\sry,\srz), T_n(\srxk^2, \sry,\srz)) \overset{\mathcal L_n}\rightarrow_d  \widetilde T \times \widetilde T
	\end{equation}
	for some limiting random variable $\widetilde T$ with continuous and strictly increasing distribution function. Then, the CRT threshold converges in probability to the upper quantile of $\widetilde T$:
	\begin{equation}
		C_n(\sry, \srz) \equiv Q_{1-\alpha}[T_n(\srxk,\sry,\srz)|\sry,\srz]  \overset{\mathcal L_n}\rightarrow_p Q_{1-\alpha}[\widetilde T].
	\end{equation}
	
\end{lemma}

\begin{lemma}\label{lem:aux}
	
	Fix any $c_1, c_2 > 0$. For any sequence
	\begin{equation}
		\mathcal L_n \in \mathscr L^{\textnormal{MX(2)}}(\mu_n(\cdot), \Sigma_n(\cdot)) \cap \mathscr L_n(c_1, c_2),
	\end{equation}
	we have
	\begin{equation}
		\widehat S_n^2 - S_n^2 \overset{\mathcal L_n}\rightarrow_p  0
		\label{eq:conv-s-2}
	\end{equation}
	and
	\begin{equation}
		\widehat S_n^{-1}S_n \overset{\mathcal L_n}\rightarrow_p I_d.
		\label{eq:conv-s-neg-1}
	\end{equation}
	
	\begin{proof}
		To show the first convergence~\eqref{eq:conv-s-2}, we apply the WLLN to the triangular array $\{(\sry_{i} - \widehat g_n(\srz_{i}))^2\Sigma_n(\srz_i)\}_{i,n}$. We first verify the second moment condition:
		\begin{equation}
			\begin{split}
				&\sup_{n}\ \mathbb E_{\mathcal L_n}[\|(\pry - \widehat g_n(\prz))^2\Sigma_n(\prz)\|^2] \\
				&\quad= \sup_{n}\ \mathbb E_{\mathcal L_n}[(\pry - \widehat g_n(\prz))^4\|\Sigma_n(\prz)\|^2] \\
				&\quad\leq \sup_{n}\ \mathbb E_{\mathcal L_n}[(\pry - \widehat g_n(\prz))^4\mathbb E_{\mathcal L_n}[\|\prx - \mu_n(\prz)\|^2|\prz]^2] \\
				&\quad\leq \sup_{n}\ \mathbb E_{\mathcal L_n}[(\pry - \widehat g_n(\prz))^4\mathbb E_{\mathcal L_n}[\|\prx - \mu_n(\prz)\|^4|\prz]] \\
				&\quad \leq  c_2 < \infty.
			\end{split}
			\label{eq:eighth-moment-calculation}
		\end{equation} 
		Therefore, by the WLLN we obtain the convergence
		\begin{equation}
			\widehat S_n^{2} - S_n^2 =  \frac{1}{n}\sum_{i = 1}^n (\sry_{i} - \widehat g_n(\srz_{i}))^2\Sigma_n(\srz_i) - \mathbb E_{\mathcal L_n}[(\pry - \widehat g_n(\prz))^2\Sigma_n(\prz)]\overset{\mathcal L_n}\rightarrow_p 0.
		\end{equation}
		
		To show the second convergence~\eqref{eq:conv-s-neg-1}, note first that 
		\begin{equation}
			\begin{split}
				\sup_n\ \|S_n^2\| &= \sup_n\ \|\mathbb E_{\mathcal L_n}[(\pry - \widehat g_n(\prz))^2\Sigma_n(\prz)]\| \\
				&\leq  \sup_n\ \mathbb E_{\mathcal L_n}[\|(\pry - \widehat g_n(\prz))^2\Sigma_n(\prz)\|^2]^{1/2} \leq c_2^{1/2},
			\end{split}
		\end{equation}
		the last step having been derived in equation~\eqref{eq:eighth-moment-calculation}. Therefore, for every $n$, we have
		\begin{equation}
			S_n^2 \in \mathcal S \equiv \{S^2: \|S^{-1}\| \leq c_1, \|S^2\| \leq c_2^{1/2}\}.
		\end{equation}
		Since $\mathcal S$ is a compact subset of the open set of positive definite matrices, there exists a $\delta > 0$ such that
		$\mathcal S_\delta = \{S^2: \|S^2-S_0^2\| \leq \delta \text{ for some } S_0^2 \in \mathcal S\}$ is also a compact subset of the set of positive definite matrices. Since the function $S^2 \mapsto S^{-1}$ is continuous on the compact set $\mathcal S_\delta$, it must be uniformly continuous on this set as well. Fix $\gamma > 0$. By uniform continuity, there exists an $\eta > 0$ such that $\|S^2_1 - S^2_2\| \leq \eta$ implies that $\|S^{-1}_1-S^{-1}_2\| \leq \gamma$ for all $S_1^2, S_2^2 \in \mathcal S_\delta$. We therefore have that
		\begin{equation*}
			\begin{split}
				\mathbb P_{\mathcal L_n}[\|\widehat S_n^{-1} - S_n^{-1}\| > \gamma] &= \mathbb P_{\mathcal L_n}[\|\widehat S_n^{-1} - S_n^{-1}\| > \gamma, \widehat S_n^2 \in \mathcal S_\delta] +  \mathbb P_{\mathcal L_n}[\|\widehat S_n^{-1} - S_n^{-1}\| > \gamma, \widehat S_n^2 \not \in \mathcal S_\delta] \\
				&\leq \mathbb P_{\mathcal L_n}[\|\widehat S_n^{2} - S_n^{2}\| > \eta] +  \mathbb P_{\mathcal L_n}[\widehat S_n^2 \not \in \mathcal S_\delta] \\
				&\leq \mathbb P_{\mathcal L_n}[\|\widehat S_n^{2} - S_n^{2}\| > \eta] +  \mathbb P_{\mathcal L_n}[\|\widehat S_n^2-S_n^2\| > \delta].
			\end{split}
		\end{equation*}
		Using the convergence~\eqref{eq:conv-s-2}, we find that the last expression tends to zero as $n \rightarrow \infty$, from which it follows that $\mathbb P_{\mathcal L_n}[\|\widehat S_n^{-1} - S_n^{-1}\| > \gamma]  \rightarrow 0$ as $n \rightarrow \infty$. Therefore,
		\begin{equation*}
			\widehat S_n^{-1} - S_n^{-1}  \overset{\mathcal L_n}\rightarrow_p 0.
		\end{equation*}
		Multiplying this relation on the right by the bounded quantity $S_n$, we arrive at the statement~\eqref{eq:conv-s-neg-1}, which concludes the proof.
	\end{proof}
	
\end{lemma}

\begin{lemma}\label{lem:clt}
	Consider generating $(\srxk^1, \srxk^2, \sry, \srz)$ according to~\eqref{eq:two-resamples} for a sequence of laws
	\begin{equation}
		\mathcal L_n \in \mathscr L^{\textnormal{MX(2)}}(\mu_n(\cdot), \Sigma_n(\cdot)) \cap \mathscr L_n(c_1, c_2).
	\end{equation}
	We have 
	\begin{equation}
		n^{-1/2}
		\begin{pmatrix}
			S_n^{-1} & 0 \\
			0 & S_n^{-1}
		\end{pmatrix}\sum_{i = 1}^n (\sry_{i} - \widehat g_n(\srz_{i})){\srxk^1_{i} - \mu_n(\srz_i) \choose \srxk^2_{i} - \mu_n(\srz_i)} \overset{\mathcal L_n}\rightarrow_d N\left({0 \choose 0}, \begin{pmatrix}
			I_d & 0 \\
			0 & I_d
		\end{pmatrix}
		\right)
		\label{convergence-1}
	\end{equation}
	and
	\begin{equation}
		{U_n(\srxk^1, \sry, \srz) \choose U_n(\srxk^2, \sry, \srz)} \overset{\mathcal L_n}\rightarrow_d N\left({0 \choose 0},
		\begin{pmatrix}
			I_d & 0 \\
			0 & I_d
		\end{pmatrix}
		\right).
		\label{convergence-2}
	\end{equation}
\end{lemma}

\begin{proof}[Proof of Lemma~\ref{lem:clt}]
	
	Note that
	\begin{equation*}
		{U_n(\srxk^1, \sry, \srz) \choose U_n(\srxk^2, \sry, \srz)} = \begin{pmatrix}
			\widehat S_n^{-1}S_n & 0 \\
			0 & \widehat S_n^{-1}S_n
		\end{pmatrix} \cdot n^{-1/2}
		\begin{pmatrix}
			S_n^{-1} & 0 \\
			0 & S_n^{-1}
		\end{pmatrix}\sum_{i = 1}^n(\sry_{i} - \widehat g_n(\srz_{i})) {\srxk^1_{i} - \mu_n(\srz_i) \choose \srxk^2_{i} - \mu_n(\srz_i)}.
	\end{equation*}
	By Lemma~\ref{lem:aux}, we have that $\widehat S_n^{-1}S_n \overset{\mathcal L_n}\rightarrow_p I_d$, so by Slutsky we find that the second statement~\eqref{convergence-2} follows from the first~\eqref{convergence-1}. Therefore, it suffices to prove the latter convergence. To this end, we apply the CLT to the triangular array of vectors
	\begin{equation}
		\left\{ (\sry_{i} - \widehat g_n(\srz_{i}))\begin{pmatrix}
			S_n^{-1} & 0 \\
			0 & S_n^{-1}
		\end{pmatrix}{\srxk^1_{i} - \mu_n(\srz_i) \choose \srxk^2_{i} - \mu_n(\srz_i)}\right\}_{i,n}. 
	\end{equation}
	To apply the CLT, we first verify the Lyapunov condition with $\delta = 1$: 
	\begin{equation}
		\begin{split}
			&\sup_{n}\ \mathbb E_{\mathcal L_n}\left[\left\|(\pry- \widehat g_n(\prz))
			\begin{pmatrix}
				S_n^{-1} & 0 \\
				0 & S_n^{-1}
			\end{pmatrix}{\prxk^1 - \mu_n(\prz) \choose \prxk^2 - \mu_n(\prz)}\right\|^3\right] \\
			&\quad \leq \sup_{n}\ \|S_n^{-1}\|^3 \mathbb E_{\mathcal L_n}\left[|\pry- \widehat g_n(\prz)|^3(\|\prxk^1 - \mu_n(\prz)\|^2 + \|\prxk^2 - \mu_n(\prz)\|^2)^{3/2}\right] \\
			&\quad \leq \sup_{n}\ \|S_n^{-1}\|^3 \mathbb E_{\mathcal L_n}\left[|\pry- \widehat g_n(\prz)|^3C\left(\|\prxk^1 - \mu_n(\prz)\|^3 + \|\prxk^2 - \mu_n(\prz)\|^3\right)\right] \\
			&\quad = 2C\sup_{n}\ \|S_n^{-1}\|^3 \mathbb E_{\mathcal L_n}\left[|\pry- \widehat g_n(\prz)|^3\|\prxk^1 - \mu_n(\prz)\|^3\right] \\
			&\quad \leq 2C\sup_{n}\ \|S_n^{-1}\|^3 \mathbb E_{\mathcal L_n}\left[(\pry- \widehat g_n(\prz))^4\|\prxk^1 - \mu_n(\prz)\|^4\right]^{3/4} \\
			&\quad = 2C\sup_{n}\ \|S_n^{-1}\|^3 \mathbb E_{\mathcal L_n}\left[(\pry- \widehat g_n(\prz))^4\mathbb E_{\mathcal L_n}[\|\prxk^1 - \mu_n(\prz)\|^4|\prz]\right]^{3/4} \\
			&\quad = 2C\sup_{n}\ \|S_n^{-1}\|^3 \mathbb E_{\mathcal L_n}\left[(\pry- \widehat g_n(\prz))^4\mathbb E_{\mathcal L_n}[\|\prx - \mu_n(\prz)\|^4|\prz]\right]^{3/4} \\
			&\quad \leq 2C c_1^3 c_2^{3/4} < \infty.
		\end{split}
	\end{equation}
	Here $C$ is chosen such that $(a+b)^{3/2} \leq C(a^{3/2} + b^{3/2})$ for all $a,b \geq 0$.
	Next, it is easy to verify that
	\begin{equation}
		\mathbb E_{\mathcal L_n}\left[(\pry - \widehat g_n(\prz))\begin{pmatrix}
			S_n^{-1} & 0 \\
			0 & S_n^{-1}
		\end{pmatrix}{\prxk^1 - \mu_n(\prz) \choose \prxk^2 - \mu_n(\prz)}\right] = {0 \choose 0}
	\end{equation}
	and
	\begin{equation}
		\text{Var}_{\mathcal L_n}\left[(\pry - \widehat g_n(\prz))\begin{pmatrix}
			S_n^{-1} & 0 \\
			0 & S_n^{-1}
		\end{pmatrix}{\prxk^1 - \mu_n(\prz) \choose \prxk^2 - \mu_n(\prz)}\right] = \begin{pmatrix}
			I_d & 0 \\
			0 & I_d
		\end{pmatrix}.
	\end{equation}
	By the CLT, the convergence~\eqref{convergence-1} now follows.
\end{proof}

\begin{lemma} \label{lem:consistency}

In the setting of Theorem~\ref{thm:power}(a), define 
\begin{equation*}
\rho_n \equiv \mathbb E_{\mathcal L_n}[\textnormal{Cov}_{\mathcal L_n}[\prx,\pry|\prz]] = \overline \Sigma_n \beta \quad \text{and} \quad \rho \equiv \lim_{n \rightarrow \infty} \rho_n = \overline \Sigma \beta.
\end{equation*}
Under the assumptions of Theorem~\ref{thm:power}, the estimator $\widehat \rho_n$ defined in \eqref{rho-hat} is consistent for $\rho$:
	\begin{equation}
	\widehat \rho_n \overset{\mathcal L_n}\rightarrow _p \rho = \overline \Sigma \beta.
	\label{eq:estimation-consistency}
\end{equation}

\end{lemma}

\begin{proof}

Note that $\widehat \rho_n$ is the mean of i.i.d. terms with expectation 
\begin{equation}
	\begin{split}
		&\mathbb E_{\mathcal L_n}[(\pry - \widehat g_n(\prz))(\prx - \mu_n(\prz))]  \\
		&\quad = \mathbb E_{\mathcal L_n}[((\prx - \mu_n(\prz))^T\beta + \peps + \bar g_n(\prz) - \widehat g_n(\prz))(\prx - \mu_n(\prz))] = \overline \Sigma_n \beta.
	\end{split}
\end{equation}
These terms also have bounded second moment, since
\begin{equation}
	\begin{split}
		&\mathbb E_{\mathcal L_n}[\|(\pry - \widehat g_n(\prz))(\prx - \mu_n(\prz))\|^2] \\
		&\quad= 	
		\mathbb E_{\mathcal L_n}[((\prx - \mu_n(\prz))^T\beta + \peps + \bar g_n(\prz) - \widehat g_n(\prz))^2\|\prx - \mu_n(\prz)\|^2] \\
		&\quad\leq
		C\mathbb E_{\mathcal L_n}[(\|\prx - \mu_n(\prz)\|^2\|\beta\|^2 + \peps^2 + (\bar g_n(\prz) - \widehat g_n(\prz))^2)\|\prx - \mu_n(\prz)\|^2] \\
		&\quad=
		C\|\beta\|^2\mathbb E_{\mathcal L_n}[\|\prx - \mu_n(\prz)\|^4] + C\sigma^2\mathbb E_{\mathcal L_n}[\|\prx - \mu_n(\prz)\|^2] \\
		&\quad \quad + C\mathbb E_{\mathcal L_n}[(\bar g_n(\prz) - \widehat g_n(\prz))^2\|\prx - \mu_n(\prz)\|^2].
	\end{split}
\end{equation}	
Here, $C$ is a constant so that $(a+b+c)^2 \leq C(a^2 + b^2 + c^2)$ for any $a, b, c \geq 0$.
Taking a supremum over $n$ and using the assumptions~\eqref{eq:eighth-moment-assump-1} and~\eqref{eq:eighth-moment-assump-2} yields
\begin{equation}
	\sup_n\ \mathbb E_{\mathcal L_n}[\|(\pry - \widehat g_n(\prz))(\prx - \mu_n(\prz))\|^2] < \infty.
\end{equation}
Therefore, the weak law of large numbers implies that
\begin{equation}
	\widehat \rho_n - \overline \Sigma_n \beta \overset{\mathcal L_n}\rightarrow _p  0,
\end{equation}
from which the statement~\eqref{eq:estimation-consistency} follows by the assumed convergence $\overline \Sigma_n \rightarrow \overline \Sigma$. 
	
\end{proof}

\begin{lemma} \label{lem:fourth-moment}
	In the setting of Theorem~\ref{thm:power}, define
	\begin{equation}
		\pry' \equiv \bar g_n(\prz) + \peps, \quad S'^2_{n} \equiv \mathbb E_{\mathcal L_n}[(\pry' - \widehat g_n(\prz))^2 \Sigma_n(\prz)], \quad \text{and} \quad \mathcal L'_n \equiv \mathcal L_n(\prx, \pry', \prz).
		\label{prime-definitions}
	\end{equation}
	Under the assumptions of Theorem~\ref{thm:power}(a) or~\ref{thm:power}(b), 
	\begin{equation}
		\text{there exist } c_1, c_2 > 0 \text{ such that } \mathcal L_n, \mathcal L'_n \in \mathscr L(c_1, c_2).
		\label{eighth-moment}
	\end{equation} 
	Under the assumptions of Theorem~\ref{thm:power}(a), we have
	\begin{equation}
	\inf_{n}\ \lambda_{\min}(S_n^{-1}) > 0,
	\label{eq:s-n-2-limit-a}
	\end{equation}
while under the assumptions of Theorem~\ref{thm:power}(b), we have
	\begin{equation}
		\lim_{n \rightarrow \infty} S_n^2 = \lim_{n \rightarrow \infty}  S'^2_n = \overline \Sigma^{1/2}(\sigma^2 I_d + \mathcal E^2)\overline \Sigma^{1/2}.
		\label{eq:s-n-2-limit}
	\end{equation}
	
\end{lemma}
\begin{proof}
First, we show that under the assumptions of Theorem~\ref{thm:power}(a) or~\ref{thm:power}(b), we have $\mathcal L_n \in \mathscr L(c_1, c_2)$ for some $c_1, c_2 > 0$. It suffices to show that
\begin{equation}
\sup_{n} \|S_n^{-1}\| < \infty
\label{eq:bounded-inverse}
\end{equation}
and
\begin{equation}
\sup_n\ \mathbb E_{\mathcal L_n}\left[(\pry - \widehat g_n(\prz))^{4} \mathbb E_{\mathcal L_n}[\|\prx - \mu_n(\prz)\|^{4}|\prz]\right] < \infty.
\label{eq:bounded-eighth-moment}
\end{equation}
To show the statement~\eqref{eq:bounded-inverse}, first note that
\begin{equation}
	\begin{split}
		S_n^2 &= \mathbb E_{\mathcal L_n}[(\pry - \widehat g_n(\prz))^2 \Sigma_n(\prz)] \\
		&=   \mathbb E_{\mathcal L_n}[((\prx - \mu_n(\prz))^T \beta_n + \pry'-\widehat g_n(\prz))^2 \Sigma_n(\prz)] \\
		&= \mathbb E_{\mathcal L_n}[((\prx - \mu_n(\prz))^T \beta_n)^2\Sigma_n(\prz)]  \\
		&\quad+2\mathbb E_{\mathcal L_n}[(\prx - \mu_n(\prz))^T \beta_n( \pry'-\widehat g_n(\prz))\Sigma_n(\prz)] \\
		&\quad + \mathbb E_{\mathcal L_n}[( \pry'-\widehat g_n(\prz))^2 \Sigma_n(\prz)] \\
		&= \mathbb E_{\mathcal L_n}[((\prx - \mu_n(\prz))^T \beta_n)^2\Sigma_n(\prz)] + \mathbb E_{\mathcal L_n}[(\pry'-\widehat g_n(\prz))^2 \Sigma_n(\prz)],
		\label{eq:s-n-2}
	\end{split}
\end{equation}
where in the last step we used the fact that
\begin{equation*}
	\begin{split}
		&\mathbb E_{\mathcal L_n}[(\prx - \mu_n(\prz))^T \beta_n( \pry'-\widehat g_n(\prz))\Sigma_n(\prz)] \\
		&\quad=\mathbb E_{\mathcal L_n}[\mathbb E_{\mathcal L_n}[(\prx - \mu_n(\prz))|\prz]^T \beta_n( \bar g_n(\prz)-\widehat g_n(\prz))\Sigma_n(\prz)] = 0. \\
	\end{split}
\end{equation*}
Furthermore, 
\begin{equation}
	\begin{split}
		S'^2_n &= \mathbb E_{\mathcal L_n}[(\pry'-\widehat g_n(\prz))^2 \Sigma_n(\prz)] \\
		&=\mathbb E_{\mathcal L_n}[( \peps + \bar g_n(\prz)-\widehat g_n(\prz))^2 \Sigma_n(\prz)] \\
		& = \mathbb E_{\mathcal L_n}[\peps ^2 \Sigma_n(\prz)] + \mathbb E_{\mathcal L_n}[2\peps(\bar g_n(\prz)-\widehat g_n(\prz))\Sigma_n(\prz)]  \\
		&\quad + \mathbb E_{\mathcal L_n}[(\bar g_n(\prz)-\widehat g_n(\prz))^2\Sigma_n(\prz)] \\
		&= \sigma^2 \overline \Sigma_n + \overline \Sigma_n^{1/2}\mathcal E^2_n \overline \Sigma_n^{1/2}  = \overline \Sigma_n^{1/2}(\sigma^2 I_d + \mathcal E^2_n) \overline \Sigma_n^{1/2}.
		\label{eq:s-n-2-prime}
	\end{split}
\end{equation}
It follows that $S_n^2 \succcurlyeq \sigma^2 \overline \Sigma_n$, which together with assumption~\eqref{eq:s-n-inverse-assump} implies that
\begin{equation}
\sup_n \|S_n^{-1}\| \leq \sup_n \|\sigma^{-1}\overline \Sigma_n^{-1/2}\| = \sigma^{-1}\left(\sup_n \|\overline \Sigma_n^{-1}\|\right)^{1/2} < \infty.
\end{equation}
This verifies statement~\eqref{eq:bounded-inverse}.	To prove statement~\eqref{eq:bounded-eighth-moment}, we write
	\begin{equation*}
		\begin{split}
			&\mathbb E_{\mathcal L_n}\left[ (\pry - \widehat g_n(\prz))^{4}\mathbb E_{\mathcal L_n}[\|\prx - \mu_n(\prz)\|^{4}|\prz]\right] \\
			&\quad= \mathbb E_{\mathcal L_n}\left[ ((\prx - \mu_n(\prz))^T \beta_n + \peps + \bar g_n(\prz)-\widehat g_n(\prz))^4\mathbb E_{\mathcal L_n}[\|\prx - \mu_n(\prz)\|^{4}|\prz]\right] \\
			&\quad\leq C\mathbb E_{\mathcal L_n}\left[(((\prx - \mu_n(\prz))^T \beta_n)^4 + \peps^4 + (\bar g_n(\prz)-\widehat g_n(\prz))^4) \mathbb E_{\mathcal L_n}[\|\prx - \mu_n(\prz)\|^{4}|\prz]\right] \\
			&\quad\leq C\|\beta_n\|^4\mathbb E_{\mathcal L_n}[\|\prx - \mu_n(\prz)\|^{8}] + 3C\sigma^4\mathbb E_{\mathcal L_n}[\|\prx - \mu_n(\prz)\|^{4}] \\
			&\quad \quad + C \mathbb E_{\mathcal L_n}\left[(\bar g_n(\prz) - \widehat g_n(\prz))^{4} \|\prx - \mu_n(\prz)\|^{4}\right].
		\end{split}
	\end{equation*}
	Here, $C$ a constant such that $(a + b + c)^4 \leq C(a^4 + b^4 + c^4)$ for all $a,b,c \geq 0$. Taking a supremum over $n$ and using the moment assumptions~\eqref{eq:eighth-moment-assump-1} and \eqref{eq:eighth-moment-assump-2} along with the boundedness of the sequence $\beta_n$ yields the statement~\eqref{eq:bounded-eighth-moment}. 
	
	Therefore, we have verified that $\mathcal L_n \in \mathscr L(c_1, c_2)$ for some $c_1, c_2$ under the assumptions of Theorem~\ref{thm:power}(a) or~\ref{thm:power}(b). The fact that $\mathcal L'_n \in \mathscr L(c_1, c_2)$ under these assumptions follows by a similar argument (omitted for the sake of brevity), which finishes the proof of statement \eqref{eighth-moment}.
	
	Next, we turn to proving the claim~\eqref{eq:s-n-2-limit-a}. Using calculations~\eqref{eq:s-n-2} and~\eqref{eq:s-n-2-prime}, we write
	\begin{equation}
	S_n^2 = \mathbb E_{\mathcal L_n}[((\prx - \mu_n(\prz))^T \beta)^2\Sigma_n(\prz)] + \overline \Sigma_n^{1/2}(\sigma^2 I_d + \mathcal E^2_n) \overline \Sigma_n^{1/2}.
\end{equation}
	Note that 
	\begin{equation}
		\begin{split}
			&\sup_n\ \|\mathbb E_{\mathcal L_n}[((\prx - \mu_n(\prz))^T \beta)^2\Sigma_n(\prz)]\| \\
			&\quad\leq \sup_n\ \|\beta\|^2\mathbb E_{\mathcal L_n}[\|\prx - \mu_n(\prz)\|^2\mathbb E_{\mathcal L_n}[\|\prx - \mu_n(\prz)\|^2|\prz]] \\
			&\quad= \sup_n\ \|\beta\|^2\mathbb E_{\mathcal L_n}[\mathbb E_{\mathcal L_n}[\|\prx - \mu_n(\prz)\|^2|\prz]^2] \\
			&\quad\leq \sup_n\ \|\beta\|^2\mathbb E_{\mathcal L_n}[\mathbb E_{\mathcal L_n}[\|\prx - \mu_n(\prz)\|^4|\prz]] \\
			&\quad\leq \sup_n\ \|\beta\|^2\mathbb E_{\mathcal L_n}[\|\prx - \mu_n(\prz)\|^4] < \infty, \\
			\label{eq:boundedness}
		\end{split}
	\end{equation}
	the last step using the eighth moment bound~\eqref{eq:eighth-moment-assump-1}. Furthermore, 
	\begin{equation}
	\sup_n \ \|\overline \Sigma_n^{1/2}(\sigma^2 I_d + \mathcal E^2_n) \overline \Sigma_n^{1/2}\| < \infty
	\end{equation}
	because $\overline \Sigma_n^{1/2}(\sigma^2 I_d + \mathcal E^2_n) \overline \Sigma_n^{1/2}$ is a convergent sequence by assumption. Hence, $\sup_n \|S_n^2\| < \infty$ and therefore
	\begin{equation*}
	\inf_n \lambda_{\min}(S_n^{-1}) = \inf_n \|S_n\|^{-1} = \inf_n \|S_n^2\|^{-1/2} = \left(\sup_n \|S_n^2\|\right)^{-1/2} > 0.
	\end{equation*}
	This completes the proof of claim~\eqref{eq:s-n-2-limit-a}.	
	
	Finally, we turn to proving claim~\eqref{eq:s-n-2-limit}. The claimed convergence of $S'^2_n$ follows immediately from the derivation~\eqref{eq:s-n-2-prime} and the assumption~\eqref{eq:limits}. To show that $S_n^2$ has the same limit, note that the derivation~\eqref{eq:s-n-2} implies that 
	\begin{equation*}
		S_n^2 - S'^2_n	= \mathbb E_{\mathcal L_n}[((\prx - \mu_n(\prz))^T \beta_n)^2\Sigma_n(\prz)] = \frac{1}{n}\mathbb E_{\mathcal L_n}[((\prx - \mu_n(\prz))^T h_n)^2\Sigma_n(\prz)].
	\end{equation*}
	The boundedness of the quantity $\mathbb E_{\mathcal L_n}[((\prx - \mu_n(\prz))^T h_n)^2\Sigma_n(\prz)]$ follows by an argument analagous to that in equation~\eqref{eq:boundedness}, which shows that
	\begin{equation*}
		S_n^2 - S'^2_n \rightarrow 0.
	\end{equation*}	
	This completes the proof of statement~\eqref{eq:s-n-2-limit}, so we are done.
\end{proof}

\section{Proofs for Section~\ref{sec:knockoffs}} \label{sec:knockoffs-proofs}

\begin{proof}[Proof of Theorem~\ref{prop:knockoff-optimality}]
	Let us denote
	\begin{equation*}
	[\srx, \srxk]_{?} \equiv (\{\srx_j, \srxk_j\}, \srx_{-j}, \srxk_{-j}),
	\end{equation*}
	where $\{\srx_j, \srxk_j\}$ represents the \textit{unordered} pair. In other words, $[\srx, \srxk]_{?}$ specifies $[\srx, \srxk]$ up to a swap, hence the ``?" notation: 
	\begin{equation*}
	[\srx, \srxk]_{?} = [\sfx, \sfxk]_{?} \quad \Longleftrightarrow \quad [\srx, \srxk] \in \{[\sfx, \sfxk], [\sfx, \sfxk]_{\text{swap}(j)}\}.
	\end{equation*}
	With this notation, we claim that
	\begin{equation}
	T_j^{\textnormal{opt}} \in \underset{T_j}{\arg \max}\ \mathbb P\left[\left. T_j([\srx, \srxk], \sry) > T_j([\srx, \srxk]_{\text{swap}(j)}, \sry)\ \right|\ [\srx, \srxk]_? = [\sfx, \sfxk]_?, \sry = \sfy\right] 
	\label{knockoff-conditional-optimality}
	\end{equation}
	for every $([\sfx, \sfxk], \sfy)$ in the set 
	\begin{equation}
	\mathcal A \equiv \left\{([\sfx, \sfxk], \sfy): T^{\text{opt}}_j([\sfx, \sfxk], \sfy) \neq T^{\text{opt}}_j([\sfx, \sfxk]_{\text{swap}(j)}, \sfy)\right\}. 
	\label{nondegeneracy}
	\end{equation}
	The conclusion~\eqref{unconditional-knockoff-optimality} will follow because for any $T_j$, 
	\begin{equation*}
	\begin{split}
	&\mathbb P[T_j([\srx, \srxk], \sry) > T_j([\srx, \srxk]_{\text{swap}(j)}, \sry)] \\
	&= \mathbb P[T_j([\srx, \srxk], \sry) > T_j([\srx, \srxk]_{\text{swap}(j)}, \sry), \srx_j \neq \srxk_j] \\
	&= \mathbb P[T_j([\srx, \srxk], \sry) > T_j([\srx, \srxk]_{\text{swap}(j)}, \sry), ([\srx, \srxk], \sry) \in \mathcal A] \\
	&= \mathbb P\left[\left.T_j([\srx, \srxk], \sry) > T_j([\srx, \srxk]_{\text{swap}(j)}, \sry)\right| ([\srx, \srxk], \sry) \in \mathcal A\right]\mathbb P[([\srx, \srxk], \sry) \in \mathcal A] \\
	&= \mathbb E\left[\left.\mathbb P\left[\left.T_j([\srx, \srxk], \sry) > T_j([\srx, \srxk]_{\text{swap}(j)}, \sry)\right| [\srx, \srxk]_?, \sry\right]\right|([\srx, \srxk], \sry) \in \mathcal A\right]\mathbb P[([\srx, \srxk], \sry) \in \mathcal A] \\
	&\leq \mathbb E\left[\left.\mathbb P\left[\left.T^{\text{opt}}_j([\srx, \srxk], \sry) > T^{\text{opt}}_j([\srx, \srxk]_{\text{swap}(j)}, \sry)\right| [\srx, \srxk]_?, \sry\right]\right|([\srx, \srxk], \sry) \in \mathcal A\right]\mathbb P[([\srx, \srxk], \sry) \in \mathcal A] \\
	&=\mathbb P\left[T^{\text{opt}}_j([\srx, \srxk], \sry) > T^{\text{opt}}_j([\srx, \srxk]_{\text{swap}(j)}, \sry)\right].
	\end{split}
	\end{equation*}
	The first step holds because $T_j([\srx, \srxk], \sry) > T_j([\srx, \srxk]_{\text{swap}(j)}, \sry)$ implies that $\srx_j \neq \srxk_j$, the second by the assumption~\eqref{nondegeneracy-assumption}, the third and fourth by probability manipulations, the fifth by the claimed conditional optimality~\eqref{knockoff-conditional-optimality}, and the sixth by the same logic as the first four steps.

	To prove equation~\eqref{knockoff-conditional-optimality}, fix $([\sfx, \sfxk], \sfy) \in \mathcal A$. Consider the simple hypothesis testing problem
	\begin{equation}
	H_0: (\srx_j, \srxk_j) = (\sfxk_j, \sfx_j) \quad \text{versus} \quad H_1: (\srx_j, \srxk_j) = (\sfx_j, \sfxk_j),
	\label{knockoffs-simple}
	\end{equation}
	where $(\srx_j, \srxk_j)$ are endowed with their law conditional on 
	\begin{equation*}
	([\srx, \srxk]_?, \sry) = ([\sfx, \sfxk]_?, \sfy).
	\end{equation*}
	We seek the most powerful test of level $\alpha = 1/2$. Note that under the null distribution, the knockoff exchangeability property makes both events equally likely: $\mathbb P_0[(\srx_j, \srxk_j) = (\sfx_j, \sfxk_j)] = \mathbb P_0[(\srx_j, \srxk_j) = (\sfxk_j, \sfx_j)] = 1/2$. Therefore, given any statistic $T_j$, the level 1/2 test of the simple hypothesis~\eqref{knockoffs-simple} rejects when $T_j([\srx, \srxk], \sry) > T_j([\srx, \srxk]_{\text{swap}(j)}, \sry)$. The knockoff statistic $T^{\text{opt}}$ defined in equation~\eqref{knockoff-conditional-optimality} thus coincides with the most powerful test for the hypothesis~\eqref{knockoffs-simple}, which by Neyman-Pearson is given by
	\begin{equation*}
	\begin{split}
	&T^{\text{opt}}_j([\sfx, \sfxk], \sfy) \\
	&\quad= \frac{\mathbb P\left[\left.(\srx_j, \srxk_j) = (\sfx_j, \sfxk_j)\right|[\srx, \srxk]_? = [\sfx, \sfxk]_?, \sry = \sfy\right]}{\mathbb P\left[\left.(\srx_j, \srxk_j) = (\sfxk_j, \sfx_j)\right|[\srx, \srxk]_? = [\sfx, \sfxk]_?, \sry = \sfy\right]} \\
	&\quad= \frac{\mathbb P\left[\left.(\srx_j, \srxk_j) = (\sfx_j, \sfxk_j)\right| [\srx, \srxk]_? = [\sfx, \sfxk]_?\right]\mathbb P\left[\sry = \sfy\left|[\srx, \srxk] = [\sfx,\sfxk]\right.\right]}{\mathbb P\left[\left.(\srx_j, \srxk_j) = (\sfxk_j, \sfx_j)\right| [\srx, \srxk]_? = [\sfx, \sfxk]_?\right]\mathbb P\left[\sry = \sfy\left|[\srx, \srxk] = [\sfx,\sfxk]_{\text{swap}(j)}\right.\right]} \\
	&\quad= \frac{\mathbb P\left[\sry = \sfy\left|[\srx, \srxk] = [\sfx,\sfxk]\right.\right]}{\mathbb P\left[\sry = \sfy\left|[\srx, \srxk] = [\sfx,\sfxk]_{\text{swap}(j)}\right.\right]} = \frac{\mathbb P\left[\sry = \sfy\left|\srx_j = \sfx_j, \srx_{-j} = \sfx_{-j}\right.\right]}{\mathbb P\left[\sry = \sfy\left|\srx_j = \sfxk_j, \srx_{-j} = \sfx_{-j}\right.\right]}.
	\end{split}
	\end{equation*}
	The first step is given by Neyman-Pearson, the second by an application of Bayes rule, the third by the conditional exchangeability of knockoffs~\eqref{conditional-exchangeability}, and the last by the conditional independence of knockoffs~\eqref{knockoff-conditional-independence}. Finally, it is easy to verify that
	\begin{equation*}
	\begin{split}
	&T_j^{\text{opt}}([\srx, \srxk], \sry) > T_j^{\text{opt}}([\srx, \srxk]_{\text{swap}(j)}, \sry) \quad \Longleftrightarrow \\
	& \mathbb P[\sry = \sfy|\srx_j = \sfx_j, \srx_{-j} = \sfx_{-j}] >  \mathbb P[\sry = \sfy|\srx_j = \sfxk_j, \srx_{-j} = \sfx_{-j}],
	\end{split}
	\end{equation*}
	from which we conclude that the likelihood given in equation~\eqref{log-likelihood-ratio-knockoffs} is optimal for the problem~\eqref{conditional-optimality}. This completes the proof.
\end{proof}


\begin{proof}[Proof of Proposition~\ref{prop:nondegeneracy-knockoffs}]
	
	Suppose $\prx_{j}|\prx_{-j}, \prxk$ has a density with respect to the Lebesgue measure. Since
	\begin{equation*}
	\begin{split}
	&\mathbb P[T^{\textnormal{opt}}_j([\srx, \srxk], \sry) = T^{\textnormal{opt}}_j([\srx, \srxk]_{\textnormal{swap}(j)}, \sry), \srx_{\bullet,j} \neq \srxk_{\bullet,j}] \\
	&\quad= \mathbb E[\mathbb P[T^{\textnormal{opt}}_j([\srx, \srxk], \sry) = T^{\textnormal{opt}}_j([\srx, \srxk]_{\textnormal{swap}(j)}, \sry), \srx_{\bullet,j} \neq \srxk_{\bullet,j}\ |\ \srx_{\bullet,-j}, \sry, \srxk]],
	\end{split}
	\end{equation*}
	it suffices to show that 
	\begin{equation*}
	\mathbb P[T^{\textnormal{opt}}_j([\srx, \srxk], \sry) = T^{\textnormal{opt}}_j([\srx, \srxk]_{\textnormal{swap}(j)}, \sry)\ |\ \srx_{\bullet,-j}, \sry, \srxk] = 0
	\end{equation*}
	for all $\srx_{\bullet,-j}, \sry, \srxk_j$. Since $\mathcal L(\prx_j|\prx_{-j}, \prxk)$ has a density with respect to the Lebesgue measure, so do $\mathcal L(\prx_j|\pry, \prx_{-j}, \prxk)$ and $\mathcal L(\srx_j|\sry,  \srx_{\bullet,-j}, \srxk)$. Therefore, it suffices to show that the set
	\begin{equation*}
	S(c; \sfx_{\bullet,-j}, \sfy) \equiv \{x_{\bullet,j} : \mathbb P(\sry = \sfy|\srx_{\bullet,j} = \sfx_{\bullet,j}, \srx_{\bullet,-j} = \sfx_{\bullet,-j}) = c\} \subseteq \mathbb R^{n}
	\end{equation*}
	has Lebesgue measure zero for all $c, \sfx_{\bullet,-j}, \sfy$. To see this, note that if $\sfx_{\bullet,j} \in S(c; \sfx_{\bullet,-j}, \sfy)$, then
	\begin{equation*}
	\begin{split}
	c &= \mathbb P(\sry = \sfy|\srx_{\bullet,j} = \sfx_{\bullet,j}, \srx_{\bullet,-j} = \sfx_{\bullet,-j})\\
	&= \prod_{i = 1}^n \exp(\eta_i \sfy_i - \psi(\eta_i))g_0(\sfy_i) \\
	&= \exp\left(\sum_{i = 1}^n (x_{ij}\beta_j  + f_{-j}(\sfx_{i,-j}))\sfy_i - \psi(\sfx_{ij}\beta_j  + f_{-j}(\sfx_{i,-j})) + \log g_0(\sfy_i) \right).
	\end{split}
	\end{equation*}
	It follows that
	\begin{equation}
	\begin{split}
	&S(c; \sfx_{\bullet,-j}, \sfy) \\
	&= \left\{x_{\bullet,j}: \sum_{i = 1}^n [\sfx_{ij}\beta_j \sfy_i  - \psi(\sfx_{ij} \beta_j + f_{-j}(\sfx_{i,-j}))] =  \log c- \sum_{i = 1}^n [f_{-j}(\sfx_{i,-j})\sfy_i  + \log g_0(\sfy_i)]\right\}. 
	\end{split}
	\label{likelihood-expression}
	\end{equation}
	Since $\psi$ is strictly convex and $\beta_j \neq 0$, the left hand side is a strictly concave function of $x_{\bullet,j}$, while the right hand side is a constant (with respect to $x_{\bullet,j}\beta_j$). Thus, $S(c; \sfx_{\bullet,-j}, \sfy)$ 
	is the level set of a strictly concave function, and hence has measure zero. Indeed, the level set of a strictly convex function is the boundary of the corresponding super-level set (which must be convex), and the boundary of any convex set has measure zero~\cite{Lang1986}. Thus, the  conclusion~\eqref{nondegeneracy-assumption} thus follows.
	
	
	Now, assume that $g_\eta$ has a density with respect to Lebesgue measure. Since
	\begin{equation*}
	\begin{split}
	&\mathbb P[T^{\textnormal{opt}}_j([\srx, \srxk], \sry) = T^{\textnormal{opt}}_j([\srx, \srxk]_{\textnormal{swap}(j)}, \sry), \srx_{\bullet,j} \neq \srxk_{\bullet,j}] \\
	&\quad= \mathbb E[\mathbb P[T^{\textnormal{opt}}_j([\srx, \srxk], \sry) = T^{\textnormal{opt}}_j([\srx, \srxk]_{\textnormal{swap}(j)}, \sry), \srx_{\bullet,j} \neq \srxk_{\bullet,j}\ |\ \srx, \srxk]],
	\end{split}
	\end{equation*}
	it suffices to show that
	\begin{equation}
	\mathbb P[P(\sry|\srx_{\bullet,j}, \srx_{\bullet,-j}) = P(\sry|\srxk_{\bullet,j}, \srx_{\bullet,-j})\ |\ \srx, \srxk] = 0
	\label{sufficient-knockoffs}
	\end{equation}
	for all $\srx_{\bullet,j} \neq \srxk_{\bullet,j}$. From expression~\eqref{likelihood-expression}, we see that $P(\sry|\srx_{\bullet,j}, \srx_{\bullet,-j}) = P(\sry|\srxk_{\bullet,j}, \srx_{\bullet,-j})$ if and only if
	\begin{equation*}
	\underbrace{\beta_j (\srx_{\bullet,j} - \srxk_{\bullet,j})^T}_{\text{slope}} \sry - \underbrace{\psi(\beta_j \srx_{i,j} + f_{-j}(\srx_{i,-j})) + \psi(\beta_j \srxk_{i,j} + f_{-j}(\srx_{i,-j}))}_{\text{intercept}} = 0.
	\end{equation*}
	Since $\beta_j \neq 0$ by assumption, the slope $\beta_j (\srx_{\bullet,j} - \srxk_{\bullet,j}) \neq 0$ and therefore, the set $\{\sry: P(\sry|\srx_{\bullet,j}, \srx_{\bullet,-j}) = P(\sry|\srxk_{\bullet,j}, \srx_{\bullet,-j})\}$ is a hyperplane (and hence has Lebesgue measure zero). Together with the fact that $\sry$ has a density with respect to Lebesgue measure, this implies the relation~\eqref{sufficient-knockoffs}, so the conclusion~\eqref{nondegeneracy-assumption} follows. 
\end{proof}

\end{document}